\documentclass{amsart}

\input xy
\xyoption{all}

\newcommand{\pU}{p\mathcal U}
\newcommand{\qU}{q\mathcal U}

\newcommand{\IZ}{\mathbb Z}
\newcommand{\U}{\mathcal U}
\newcommand{\V}{\mathcal V}
\newcommand{\W}{\mathcal W}

\newcommand{\Ra}{\Rightarrow}
\newcommand{\e}{\varepsilon}
\newcommand{\IN}{\mathbb N}

\newcommand{\w}{\omega}

\newcommand{\rcluwL}{{}^\circ\kern-2pt\overline{uw}L}
\newcommand{\Hs}{H\kern-2pt s}
\newcommand{\St}{\mathcal{S}t}
\newcommand{\Rpsi}{\overset{...}{\kern-2pt\psi}}
\newcommand{\Tb}{\overset{...}{\kern-2pt b}}

\newcommand{\F}{\mathcal F}
\newcommand{\A}{\mathcal A}
\newcommand{\dc}{dc}
\newcommand{\cf}{\mathrm{cf}}
\newcommand{\C}{\mathcal C}
\newcommand{\elm}{\ell^{-}}

\newtheorem{theorem}{Theorem}[section]
\newtheorem{corollary}[theorem]{Corollary}
\newtheorem{proposition}[theorem]{Proposition}
\newtheorem{problem}[theorem]{Problem}

\newtheorem{claim}[theorem]{Claim}
\newtheorem{lemma}[theorem]{Lemma}
\theoremstyle{definition}
\newtheorem{definition}[theorem]{Definition}
\newtheorem{example}[theorem]{Example}
\newtheorem{remark}[theorem]{Remark}

\title{Verbal covering properties of topological spaces}
\author{Taras Banakh and Alex Ravsky}
\address{T.Banakh: Ivan Franko National University of Lviv (Ukraine), and Jan Kochanowski University in Kielce (Poland)}
\email{t.o.banakh@gmail.com}
\address{A.Ravsky: Pidstryhach Institute for Applied Problems of Mechanics and Mathematics of National Academy of Sciences, Lviv, Ukraine}
\email{oravsky@mail.ru}
\thanks{The first author has been partially financed by NCN grant DEC-2012/07/D/ST1/02087.}

\begin{document}
\begin{abstract} For any topological space $X$ we study the relation between the universal uniformity $\U_X$, the universal quasi-uniformity $\qU_X$ and the universal pre-uniformity $\pU_X$ on $X$.
For a pre-uniformity $\U$ on a set $X$ and a word $v$ in the two-letter alphabet $\{+,-\}$ we define
the verbal power $\U^v$ of $\U$ and study its boundedness numbers $\ell(\U^v)$, $\bar \ell(\U^v)$, $L(\U^v)$ and $\bar L(\U^v)$. The boundedness numbers of (the Boolean operations over) the verbal powers of the canonical pre-uniformities $\pU_X$, $\qU_X$ and $\U_X$ yield new cardinal characteristics
$\ell^v(X)$, $\bar \ell^v(X)$,  $L^v(X)$, $\bar L^v(X)$, $q\ell^v(X)$, $q\bar \ell^v(X)$,  $qL^v(X)$, $q\bar L^v(X)$, $u\ell(X)$ of a topological space $X$, which generalize all known cardinal topological invariants related to  (star)-covering properties. We study the relation of the new cardinal invariants $\ell^v$, $\bar \ell^v$ to classical cardinal topological invariants such as Lindel\"of number $\ell$,
density $d$, and spread $s$. The simplest new verbal cardinal invariant is the foredensity $\elm(X)$ defined for a topological space $X$ as the smallest cardinal $\kappa$ such that for any neighborhood assignment $(O_x)_{x\in X}$ there is a subset $A\subset X$ of cardinality $|A|\le\kappa$ that meets each neighborhood $O_x$, $x\in X$. It is clear that $\elm(X)\le d(X)\le \elm(X)\cdot \chi(X)$. We shall prove that $\elm(X)=d(X)$ if $|X|<\aleph_\w$. On the other hand, for every singular cardinal $\kappa$ (with $\kappa\le 2^{2^{\cf(\kappa)}}$) we construct a (totally disconnected) $T_1$-space $X$ such that $\elm(X)=\cf(\kappa)<\kappa=|X|=d(X)$.
\end{abstract}

\maketitle

\section*{Introduction}

In this paper we suggest a uniform treatment of many (star-)covering properties considered in topological literature. Namely, for every word $v$ in the two-letter alphabet $\{+,-\}$ we define $v$-compact, weakly $v$-compact, $v$-Lindel\"of, and weakly $v$-Lindel\"of spaces, and the corresponding cardinal topological invariants $L^v$, $\bar L^v$, $\ell^v$, $\bar \ell^v$, which generalize many known cardinal invariants that have covering nature. In particular, $\ell^+$ and $\bar \ell^+$ coincide with the Lindel\"of and weakly Lindel\"of numbers, $\ell^{-+}$ coincide with the weak extent and $\ell^{+-+}$ coincides with the star-Lindel\"of number. The cardinal characteristics $L^v$, $\bar L^v$, $\ell^v$, $\bar \ell^v$ are defined in Section~\ref{s3}.  Sections~\ref{s1}, \ref{s2} are of preliminary character and collect  known information on covering and star-covering properties in topological spaces and pre-uniform spaces. In Section~\ref{s3n} we introduce three canonical pre-uniformities $\pU_X$, $\qU_X$, $\U_X$ on a topological space $X$ and study the inclusion relation between these pre-uniformities and their verbal powers. Section~\ref{s4} is devoted to studying the interplay between the density $d$ and the cardinal invariant $\ell^-$ called the foredensity.

\section{Preliminaries}\label{s1}

In this section we recall some information on covering properties of topological spaces.

Let $\w$ denote the set of all finite ordinals and let $\IN=\w\setminus\{0\}$ be the set of natural numbers.

For a subset $A$ of a topological space $X$ by $\overline{A}$ we denote the closure of the set $A$ in $X$.

We recall that a family $(A_i)_{i\in I}$ of subsets of a topological space $X$ is {\em discrete} if  each point $z\in X$ has a neighborhood that meets at most one set $A_i$, $i\in I$.

\subsection{Classical cardinal invariants in topological spaces} We recall that for a topological space $X$ its {\em character} $\chi(X)$ is defined as the smallest cardinal $\kappa$ such that each point $x\in X$ has a neighborhood base $\mathcal B_x$ of cardinality $|\mathcal B_x|\le\kappa$.

Next,  we recall the definitions of the basic cardinal invariants composing the famous Hodel's diagram \cite[p.15]{Hod} (see also \cite[p.225]{Eng}).
For a topological space $X$ let
\begin{itemize}
\item $w(X)=\min\{|\mathcal B|:\mathcal B$ is a base of the topology of $X\}$ be the {\em weight} of $X$;
\item $nw(X)=\min\{|\mathcal N|:\mathcal N$ is a network of the topology of $X\}$ be the {\em network weight} of $X$;
\item $d(X)=\min\{|A|:A\subset X,\;\overline{A}=X\}$ be the {\em density} of $X$;
\item $hd(X)=\sup\{d(Y):Y\subset X\}$ be the {\em hereditary density} of $X$;
\item $l(X)$, the {\em Lindel\"of number} of $X$, be the smallest cardinal $\kappa$ such that each open cover $\U$ of $X$ has a subcover $\V\subset\U$ of cardinality $|\V|\le\kappa$;
\item $hl(X)=\sup\{l(Y):Y\subset X\}$ be the {\em hereditary Lindel\"of number} of $X$;
\item $s(X)=\sup\{|D|:D$ is a discrete subspace of $X\}$ be the {\em spread} of $X$;
\item $e(X)=\sup\{|D|:D$ is a closed discrete subspace of $X\}$ be the {\em extent} of $X$;
\item $c(X)=\sup\{|\U|:\U$ is a disjoint family of non-empty open sets in $X\}$ be the {\em cellularity} of $X$.
\end{itemize}
These nine cardinal characteristics compose the Hodel diagram \cite{Hod} in which an arrow $f\to g$ indicates that $f(X)\le g(X)$ for any topological space $X$. The same convention concerns all other diagrams drawn in this paper.
$$
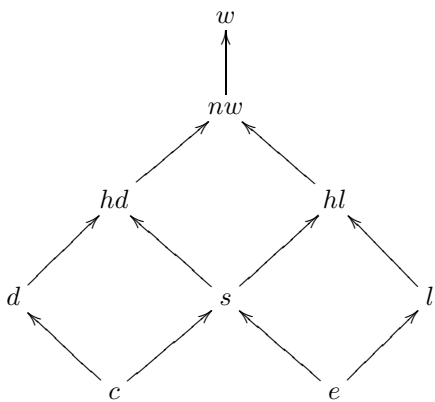
\xymatrix{
&&w\\
&&nw\ar[u]\\
&hd\ar[ru]&&hl\ar[lu]\\
d\ar[ru]&&s\ar[lu]\ar[ru]&&l\ar[lu]\\
&c\ar[lu]\ar[ru]&&e\ar[lu]\ar[ru]
}
$$
The Hodel diagram can be completed by two less known cardinal characteristics:
\begin{itemize}
\item the {\em discrete extent} $de(X)=\sup\{|\A|:\A$ is a discrete family of non-empty subsets in $X\}$ and
\item the {\em discrete cellularity} $\dc(X)=\sup\{|\U|:\U$ is a discrete family of non-empty open sets in $X\}$ of $X$.
\end{itemize}
It is easy to see that each topological space $X$ has $e(X)\le de(X)$. If $X$ is a $T_1$-space, then $de(X)=e(X)$. Therefore, the Hodel's diagram extends to the following diagram (drawn horizontally). It this diagram the arrow $de\dashrightarrow e$ indicates that $de(X)\le e(X)$ for any $T_1$-space $X$.
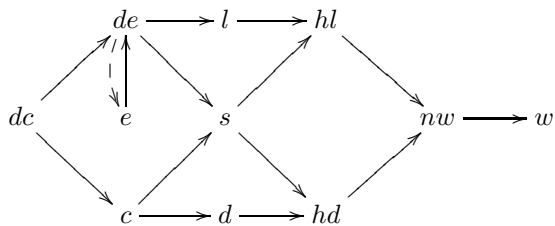
$$
\xymatrix{
&de\ar[r]\ar[rd]\ar@/_/@{-->}[d]&l\ar[r]&hl\ar[rd]\\
\dc\ar[ru]\ar[rd]&e\ar[u]&s\ar[ru]\ar[rd]&&nw\ar[r]&w\\
&c\ar[ru]\ar[r]&d\ar[r]&hd\ar[ru]
}$$

\subsection{Star-covering properties of topological spaces}

In this subsection we recall the definitions of star-covering properties introduced and studied in \cite{DRRT}, \cite{Mat}, \cite{vMTW}. For a cover $\U$ of a set $X$ and a subset $A\subset X$ we put $\St^0(A;\U)=A$ and $\St^{n+1}(A;\U)=\bigcup\{U\in\U:U\cap\St^n(A;\U)\ne\emptyset\}$ for $n\ge 0$.

For a topological space $X$ and a non-negative integer number $n\in\w$ put
\begin{itemize}
\item $L^{*n}(X)$ be the smallest cardinal $\kappa$ such that for every open cover $\U$ of $X$ there is a subset $A\subset X$ of cardinality $|A|<\kappa$ such that $\St^n(A;\U)=X$;
\item $\bar L^{*n}(X)$ be the smallest cardinal $\kappa$ such that for every open cover $\U$ of $X$ there is a subset $A\subset X$ of cardinality $|A|<\kappa$ such that $\St^n(A;\U)$ is dense in $X$;
\item $L^{*n\kern-1pt\frac12}(X)$ be the smallest cardinal $\kappa$ such that for every open cover $\U$ of $X$ there is a subfamily $\V\subset\U$ of cardinality $|\V|<\kappa$ such that $\St^n(\cup\V;\U)=X$;
\item $\bar L^{*n\kern-1pt\frac12}(X)$ be the smallest cardinal $\kappa$ such that for every open cover $\U$ of $X$ there is a subfamily $\V\subset \U$ of cardinality $|\V|<\kappa$ such that $\St^n(\cup\V;\U)$ is dense in $X$;
\smallskip
\item $l^{*n}(X)$ be the smallest cardinal $\kappa$ such that for every open cover $\U$ of $X$ there is a subset $A\subset X$ of cardinality $|A|\le\kappa$ such that $\St^n(A;\U)=X$;
\item $\bar l^{*n}(X)$ be the smallest cardinal $\kappa$ such that for every open cover $\U$ of $X$ there is a subset $A\subset X$ of cardinality $|A|\le\kappa$ such that $\St^n(A;\U)$ is dense in $X$;
\item $l^{*n\kern-1pt\frac12}(X)$ be the smallest cardinal $\kappa$ such that for every open cover $\U$ of $X$ there is a subfamily $\V\subset\U$ of cardinality $|\V|\le\kappa$ such that $\St^n(\cup\V;\U)=X$;
\item $\bar l^{*n\kern-1pt\frac12}(X)$ be the smallest cardinal $\kappa$ such that for every open cover $\U$ of $X$ there is a subfamily $\V\subset \U$ of cardinality $|\V|\le\kappa$ such that $\St^n(\cup\V;\U)$ is dense in $X$;
\item $l^{*\w}(X)=\min\{l^{*n}(X):n\in\w\}$ and $\bar l^{*\w}(X)=\min\{\bar l^{*n}:n\in\w\}$.
\end{itemize}

Observe that the cardinal characteristics $L^{*n}$, $\bar L^{*n}$, $L^{*n\kern-1pt\frac12}$, $\bar L^{*n\kern-1pt\frac12}$ determine the values of the cardinal characteristics  $l^{*n}$, $\bar l^{*n}$, $l^{*n\kern-1pt\frac12}$, $\bar l^{*n\kern-1pt\frac12}$ as $l^{*n}(X)=L^{*n}(X)_-$, $\bar l^{*n}(X)=\bar L^{*n}(X)_-$, $l^{*n\kern-1pt\frac12}(X)=L^{*n\kern-1pt\frac12}(X)_-$, $\bar l^{*n\kern-1pt\frac12}(X)=\bar L^{*n\kern-1pt\frac12}(X)_-$ for every topological space $X$. Here for a cardinal $\kappa$ by $\kappa_-=\sup\{\lambda:\lambda<\kappa\}$ we denote its ``predecessor'' (equal to $\kappa$ if $\kappa$ is a limit cardinal).

For small $n$ the cardinal characteristics $l^{*n}(X)$, $\bar l^{*n}(X)$, $l^{*n\kern-1pt\frac12}(X)$,
$\bar l^{*n\kern-1pt\frac12}(X)$ are equal to some well-known cardinal invariants. In particular,
$l^{*0}(X)=|X|$, $\bar l^{*0}(X)=d(X)$, $l^{*\frac12}(X)=l(X)$. The cardinal $\bar l^{*\frac12}(X)$ coincides with the {\em weak Lindel\"of number} $\bar l(X)$ of $X$ (called the {\em weak covering number} in \cite{Hod}) and $l^{*1}(X)$ equals the {\em weak extent} $l^*(X)$ of $X$.

Topological spaces $X$ with $L^{*n\kern-1pt\frac12}(X)\le\w$ (resp. $l^{*n\kern-1pt\frac12}(X)\le\w$) are called {\em $n$-star compact} (resp. {\em $n$-star Lindel\"of\/}), see \cite{DRRT}. Topological spaces $X$ with $\bar L^{*\kern-1pt n\kern-1pt\frac12}(X)\le\w$ (resp. $\bar l^{*\kern-1pt n\kern-1pt\frac12}(X)\le\w$) will be called {\em weakly $n$-star compact} (resp. {\em weakly $n$-star Lindel\"of\/}).

Completing the Hodel's diagram with the cardinal characteristics $l^{*n}$, $l^{*n\kern-1pt\frac12}$, $\bar l^{*n}$ and $\bar l^{*n\kern-1pt\frac12}$ we get the following diagram in which we assume that $n\ge 2$:
{
$$\xymatrix{
l^{*\w}\ar@{=}[dd]\ar[r]&\cdots\ar[r]&l^{*(n{+}1)}\ar[r]\ar[rdd]&l^{*(n{+}\frac12)}\ar[r]\ar[rdd]&l^{*n}\ar[r]
&\cdots\ar[r]&l^{*1}\ar[r]\ar[rrdd]&de\ar[r]\ar[rd]&l^{*\kern-1pt\frac12}=l\ar[r]&hl\ar[d]\\
&&&&&dc\ar[rru]\ar[rrd]&&&s\ar[r]\ar[ru]\ar[rd]&nw\\
\bar l^{*\w}\ar[r]&\cdots\ar[r]&\bar l^{*(n{+}\frac12)}\ar[r]\ar[ruu]&\bar l^{*n}\ar[r]\ar[ruu]&\bar l^{*(n{-}\frac12)}\ar[r]\ar[ru]&\cdots\ar[r]&\bar l^{*\frac12}\ar[r]\ar[rruu]&c\ar[r]\ar[ru]&\bar l^{*0}=d\ar[r]&hd\ar[u]
}
$$
}
Two non-trivial inequalities $l^{*1}\le de$ and $\bar l^{*1\kern-1pt \frac12}\le dc$ in this diagram are proved in the following proposition, in which for a topological space $X$ by $ld(X)$ we denote the {\em local density} of $X$, equal to the smallest cardinal $\kappa$ such that each point $x\in X$ has a neighborhood $O_x\subset X$ of density $d(O_x)\le\kappa$.

\begin{proposition}\label{p1.1} Each topological space $X$ has $$l^{*1}(X)\le de(X),\;\;\bar l^{*1\kern-1pt\frac12}(X)\le \dc(X) \mbox{ \ and \  }\bar l^{*1}(X)\le \dc(X)\cdot ld(X).$$
\end{proposition}

\begin{proof} To prove that  $l^{*1}(X)\le de(X)$, take any open cover $\U$ of $X$ and using Zorn's Lemma, choose a maximal subset $A\subset X$ such that $a\notin \St(b;\U)$ for any distinct points $a,b\in A$. We claim that the family of singletons $\big\{\{a\}:a\in A\big\}$ is discrete in $X$. Indeed, assuming that for some point $x\in X$ each neighborhood $O_x\subset X$ of $x$ contains two distinct points $a,b\in A$, we can take any set $U\in\U$ containing $x$, find two distinct points $a,b\in A\cap U$ and conclude that $b\in\St(a;\U)$, which contradicts the choice of the set $A$. So, the family $\big\{\{a\}:a\in A\big\}$ is discrete and hence $|A|\le de(X)$. By the maximality of $A$, for every $x\in X$ there is a point $a\in A$ such that $x\in \St(a;\U)$, which implies $X=\St(A;\U)$. This witnesses that $l^{*1}(X)\le de(X)$.
\smallskip

To prove that $\bar l^{*1\kern-1pt\frac12}(X)\le\dc(X)$ take any open cover $\U$ of $X$. Define a family $\V$ of open sets in $X$ to be {\em $\U$-separated} if
$\St(V;\U)\cap V'=\emptyset$ for any distinct sets $V,V'\in\V$. Using Zorn's Lemma, choose a maximal $\U$-separated family $\V$ of non-empty open sets in $X$ such that each set $V\in\V$ has density $d(V)\le ld(X)$ and is contained in some set $U_V\in\U$. By the maximality of $\V$ the set $\St(\cup\V;\U)$ is dense in $X$. The $\U$-separated property of $\V$ implies that the family $\V$ is discrete in $X$ (more precisely, each point $x\in X$ has a neighborhood $U\in\U$ meeting at most one set $V\in\V$) and hence $|\V|\le\dc(X)$. The family $\U'=\{U_V:V\in\V\}$ has cardinality $|\U'|\le|\V|\le\dc(X)$ and $\overline{\St(\cup\U';\U)}=X$ witnessing that $\bar l^{*1\kern-1pt\frac12}(X)\le\dc(X)$.

In each set $V\in\V$ fix a dense subset $D_V$ of cardinality $|D_V|=d(V)\le ld(X)$. Then the set $D=\bigcup_{V\in\V}D_V$ has cardinality $|D|\le|\V|\cdot ld(X)\le \dc(X)\cdot ld(X)$ and $\overline{\St(D;\U)}=\overline{\St(\bigcup\V;\U)}=X$, witnessing that $\bar l^{*1}(X)\le \dc(X)\cdot ld(X)$.
\end{proof}

Now we detect  topological spaces for which some of the cardinal characteristics $l^{*n}$, $\bar l^{*n}$, $n\in\frac12\IN$, coincide.

We recall that a topological space $X$ is called
\begin{itemize}
\item {\em quasi-regular} if every non-empty open set $U\subset X$ contains the closure $\overline{V}$ of another non-empty open set $V\subset U$;
\item {\em collectively normal} (resp. {\em collectively Hausdorff}) if for each discrete family $\F$ of (finite) subsets of $X$ there is a discrete family $(U_F)_{F\in\F}$ of open sets such that $F\subset U_F$ for all $F\in\F$;
\item {\em locally separable} if each point $x\in X$ has a separable neighborhood;
\item a {\em Moore space} if $X$ is a regular $T_1$-space  possessing a sequence of open covers $(\U_n)_{n\in\w}$ such that the family $\{\St(x;\U_n)\}_{n\in\w}$ is a neighborhood base at each point $x\in X$.
\end{itemize}
By Theorem 1.2 of \cite{Grue}, a topological space $X$ is metrizable if and only if $X$ is a collectively normal Moore space.

\begin{proposition}\label{p1.2} Let $X$ be a quasi-regular space. Then
\begin{enumerate}
\item $\dc(X)=\bar l^{*1\kern-1pt\frac12}(X)=l^{*\w}(X)$.
\item If $X$ is normal or locally separable, then $\dc(X)=\bar l^{*1}(X)$.
\item If $X$ is perfectly normal, then $\dc(X)=c(X)=\bar l^{*\kern-1pt\frac12}(X)$.
\item If $X$ is collectively Hausdorff, then $\dc(X)=de(X)=\bar l^{*1}(X)$.
\item If $X$ is paracompact, then $\dc(X)=l(X)$.
\item If $X$ is perfectly paracompact, then $\dc(X)=hl(X)$.
\item If $X$ is a Moore space, then $l^{*1}(X)=d(X)$ and $e(X)=de(X)=hd(X)$.
\end{enumerate}
\end{proposition}

\begin{proof} 1. Proposition~\ref{p1.1} implies that $l^{*\w}(X)=\bar l^{*\w}(X)\le \bar l^{*\kern-1pt 1\kern-1pt\frac12}(X)\le dc(X)$. To prove that these inequalities turn into equalities, it suffices to check that $\dc(X)\le l^{*n\kern-1pt\frac12}(X)$ for every $n\in\IN$. Assuming that $l^{*n\kern-1pt\frac12}(X)<\dc(X)$ for some $n\in\IN$, find a discrete family $\V$ of cardinality $|\V|> l^{*n\kern-1pt\frac12}(X)$ consisting of non-empty open sets in $X$. Using the quasi-regularity of $X$, for every $V\in\V$ choose a sequence of non-empty open sets $(V_i)_{i=0}^{n+1}$ such that $\overline{V_i}\subset V_j$ for every $i<j\le n+1$ and ${V}_{n+1}=V$.
Taking into account that the family $\V$ is discrete, we conclude that the set $W=X\setminus \bigcup_{V\in\V}\overline{V}_{n}$ is open in $X$. By definition of $l^{*n\kern-1pt\frac12}(X)$, for the open cover $$\W=\{W\}\cup\{V_1:V\in\V\}\cup\bigcup_{i=1}^{n}\{V_{i+1}\setminus\overline{V}_{i-1}:V\in\V\}$$ of
 $X$ there is a subfamily $\W'\subset\W$ of cardinality $|\W'|\le l^{*n\kern-1pt\frac12}(X)$ such that $X=\St^n(\cup\W';\W)$. Replacing $\W'$ by a larger subfamily, we can assume that $\W'=\{W\}\cup\{V_1:V\in\V'\}\cup\bigcup_{i=1}^{n}\{V_{i+1}\setminus\overline{V}_{i-1}:V\in\V'\}$ for some subfamily $\V'\subset\V$ of cardinality $|\V'|\le |\W'|\le l^{*n\kern-1pt\frac12}(X)<|\V|$. Choose any set $V\in\V\setminus\V'$ and observe that $V\cap \bigcup_{W'\in\W'}\St^n(W';\W)\subset \St^n(V\cap W;\W)\subset V\setminus \overline{V}_0\ne V$, which contradicts $X=\St^n(\cup\W';\W)$.
This contradiction completes the proof of the inequality $\dc(X)\le \min_{n\in\IN}l^{*n\kern-1pt\frac12}(X)=l^{*\w}(X)=\bar l^{*\w}(X)$.

The inequalities $l^{*\w}(X)\le \bar l^{*1\kern-1pt\frac12}(X)\le \dc(X)\le l^{*\w}(X)$ imply that $\bar l^{*1\kern-1pt \frac12}(X)=\dc(X)=l^{*\w}(X)$.
\smallskip

2. If the space $X$ is locally separable, then $\bar l^{*1}(X)\le \dc(X)$ by Proposition~\ref{p1.1}. Combined with the equality $\dc(X)=\bar l^{*1\kern-1pt\frac12}(X)\le \bar l^{*1}(X)$ proved in the preceding item, this yields the required equality $\bar l^{*1}(X)=\dc(X)$.
\smallskip

Next, assume that the space $X$ is normal. By the preceding item, $dc(X)=\bar l^{*\kern-1pt 1\kern-1pt\frac12}(X)\le \bar l^{*1}(X)$. To derive a contradiction, assume that  $\dc(X)<\bar l^{*\kern-1pt 1}(X)$. Then we can find an open cover $\U$ of $X$ such that for any subset $A\subset X$ of cardinality $|A|\le dc(X)$ the set $\St(A;\U)$ is not dense in $X$. Let $\kappa=\dc(X)$. By transfinite induction we can use the quasi-regularity of $X$ and construct a transfinite sequence of open sets $(V_\alpha)_{\alpha<\kappa^+}\subset X$ and points $(x_\alpha)_{\alpha<\kappa^+}$ in $X$ such that $x_\alpha\in V_\alpha\subset\overline{V}_\alpha\subset X\setminus \overline{\St(\{x_\beta\}_{\beta<\alpha};\U)}$ for every $\alpha<\kappa^+$. For every $\alpha<\kappa$ denote by $\bar x_\alpha$ the closure of the singleton $\{x_\alpha\}$ in $X$. We claim that the family  $\{\bar x_\alpha\}_{\alpha<\kappa^+}$ is discrete in $X$.
Given any point $x\in X$, we should find a neighborhood $O_x\subset X$ of $x$ that meets at most one set $\bar x_\alpha$, $\alpha<\kappa^+$. Consider the star $\St(x;\U)$ of $x$. If this star meets no set $\bar x_\alpha$, then $O_x=\St(x;\U)$ is a required neighborhood of $x$. In the opposite case we can choose the smallest ordinal $\alpha<\kappa$ such that $\bar x_\alpha\cap\St(x;\U)\ne\emptyset$. Then $x_\alpha\in \St(x;\U)$ and hence $x\in\St(x_\alpha;\U)$. For every ordinal $\beta$ with $\alpha<\beta<\kappa^+$ the choice of the point $x_\beta$ guarantees that $x_\beta\notin \St(x_\alpha;\U)$ and hence $\bar x_\beta\cap \St(x_\alpha;\U)=\emptyset$. Then the neighborhood $O_x=\St(x;\U)\cap \St(x_\alpha,\U)$ has the required property: $O_x\cap \bar x_\beta=\emptyset$ for every ordinal $\beta\in\kappa^+\setminus\{\alpha\}$.

Therefore, $\{\bar x_\alpha\}_{\alpha<\kappa^+}$ is a discrete family of closed subsets in $X$ and hence its union $F=\bigcup_{\alpha<\kappa^+}\bar x_\alpha$ is closed in $X$. Observe that for every $\alpha<\kappa^+$ the open set  $W_\alpha=\St(x_\alpha;\U)\setminus  \overline{\St(\{x_\beta\}_{\beta<\alpha};\U)}$ contains the closed set $\bar x_\alpha\subset \overline{V}_\alpha$. Consequently, $F\subset \bigcup_{\alpha<\kappa^+}W_\alpha$. By the normality of $X$, there is an open set $U\subset X$ such that $D\subset U\subset\overline{U}\subset\bigcup_{\alpha<\kappa^+}W_\alpha$. It is easy to check that the indexed family of open sets $(U\cap W_\alpha)_{\alpha<\kappa^+}$ is discrete in $X$ and hence $\dc(X)\ge \kappa^+>\kappa=\dc(X)$, which is a desired contradiction witnessing that $\dc(X)=\bar l^{*1}(X)$.
\smallskip

3. Assuming that $X$ is a perfectly normal space, we shall prove that $c(X)=\dc(X)$. To derive a contradiction, assume that $c(X)>\dc(X)$ and find a disjoint family $\U$ of non-empty open sets in $X$ with $|\U|>\dc(X)$. For each set $U\in\U$ use the quasi-regularity of $X$ to fix a non-empty open set $V_U\subset X$ such that $\overline{V}_U\subset U$. Fix any point $x_U\in V_U$ and consider the closure $\bar x_U$ of the singleton $\{x_U\}$. Then $(\bar x_U)_{U\in\U}$ is a disjoint family of closed subsets in $X$ and its union $D=\bigcup_{U\in\U}\bar x_U$ is open in the closure $\overline{D}$ of $D$ in $X$. Since $X$ is perfectly normal, $\overline{D}\setminus D$ is a $G_\delta$-set in $X$. Consequently, $D$ is an $F_\sigma$-set, which allows us to find a closed subset $F\subset D$ in $X$ such that the family $\U'=\{U\in\U:x_U\in F\}$ has cardinality $|\U'|>\dc(X)$. Then the set $D'=\bigcup_{U\in\U'}\bar x_U\subset F$ is closed in $X$ and by normality of $X$, we can find an open set $V$ in $X$ such that $F\subset V\subset\overline{V}\subset\bigcup\U$. It can be shown that $\{V\cap U:U\in\U'\}$ is a discrete collection of open sets in $X$, which implies $\dc(X)\ge |F|>\dc(X)$. This contradiction completes the proof of the equality $\dc(X)=c(X)$. Since $\dc(X)\le\bar l^{*\frac12}(X)\le c(X)$, we get $\dc(X)=\bar l^{*\frac12}(X)=c(X)$.
\smallskip

4. Assume that the space $X$ is collectively Hausdorff. Since $\dc(X)\le l^{*\kern-1pt1}(X)\le de(X)$, it suffices to show that $de(X)\le\dc(X)$. To derive a contradiction, assume that $de(X)>\dc(X)$. Then we can find a discrete family $\F$ of non-empty subsets in $X$ of cardinality $|\F|>\dc(X)$. In each set $F\in\F$ choose a point $x_F\in F$. Since the space $X$ is collectively Hausdorff, each point $x_F$ has an open neighborhood $O(x_F)\subset X$ such that the family $\U=\{O(x_F)\}_{F\in\F}$ is discrete, which implies $\dc(X)\ge|\U|\ge|\F|>\dc(X)$. This contradiction completes the proof.
\smallskip

5. Next, assuming that the space $X$ is paracompact, we shall prove that $\dc(X)=l(X)$. Since $\dc(X)\le l(X)$, it suffices to prove that $l(X)\le \dc(X)$. Fix an open cover $\U$ of $X$. Applying Theorem 5.1.12 of \cite{Eng} three times, find an open cover $\V$ of $X$ such that for every $V\in\V$ the set $\St^3(V;\V)$ is contained in some set $U\in\U$. By Zorn's Lemma, choose a maximal subset $A\subset X$ such that $y\notin \St^3(x;\V)$ for any distinct points $x,y\in A$.
We claim that the family $(\St(x;\V))_{x\in A}$ is discrete in $X$. Assuming the converse, we would find a point $x\in X$ whose each neighborhood $O_x$ meets at least two sets of the family $(\St(a;\V))_{a\in A}$. Fix any set $V\in\V$ containing $x$ and find two distinct points $a,b\in A$ such that $V\cap\St(a;\V)\ne\emptyset\ne V\cap \St(b;\V)$. Then $b\in \St^3(a;\V)$, which contradicts the choice of the set $A$. This contradiction proves that the family $(\St(a,\V))_{a\in A}$ is discrete and hence $|A|\le\dc(X)$. By the maximality of $A$, for every point $x\in X$ there is a point $a\in A$ such that $x\in \St^3(a;\V)$. Then $X=\bigcup_{a\in A}\St^3(a;\V)$. By the choice of $\V$ for every $a\in A$ there is a set $U_a\in\U$ containing the set $\St^3(a;\V)$. Then $\{U_a:a\in A\}\subset \U$ is a subcover of
cardinality $\le |A|\le\dc(X)$, witnessing that $l(X)\le\dc(X)$.
\smallskip

6. If $X$ is perfectly paracompact, then $hl(X)=l(X)\le dc(X)$ by the preceding item.
\smallskip

7. Assume that $X$ is a Moore space and fix a sequence of open covers $(\U_n)_{n\in\w}$ of $X$ such that for every $x\in X$ the family $\{\St(x;\U_n)\}_{n\in\w}$ is a neighborhood base at $x$. For every $n\in\w$ fix a subset $A_n\subset X$ of cardinality $|A_n|\le l^{*1}(X)$ such that $X=\St(A_n;\U_n)$. Then $A=\bigcup_{n\in\w}A_n$ is a dense subset of cardinality $|A|\le l^{*1}(X)$ in $X$, witnessing that $d(X)\le l^{*1}(X)$. The reverse inequality $l^{*1}(X)\le d(X)$ holds for any (not necessarily Moore) spaces.

Next, we show that every subspace $Y\subset X$ has density $d(Y)\le e(X)$. Using Zorn's Lemma, for every $n\in\w$ fix a maximal subset $D_n\subset Y$ such that $x\notin\St(y;\U_n)$ for any distinct points $x,y\in D_n$. It follows that $D_n$ is closed discrete subspace of $X$ and hence $|D_n|\le e(X)$. Moreover, $D=\bigcup_{n\in\w}D_n$ is a dense subset of $Y$, witnessing that $hd(X)\le e(X)=de(X)$.
\end{proof}

Proposition~\ref{p1.2} implies that for quasi-regular spaces the (infinite) diagram describing the relations between the cardinal characteristics $l^{*n}$, $\bar l^{*n}$, $n\in\frac12\IN$, simplifies to the following (finite) form:
$$
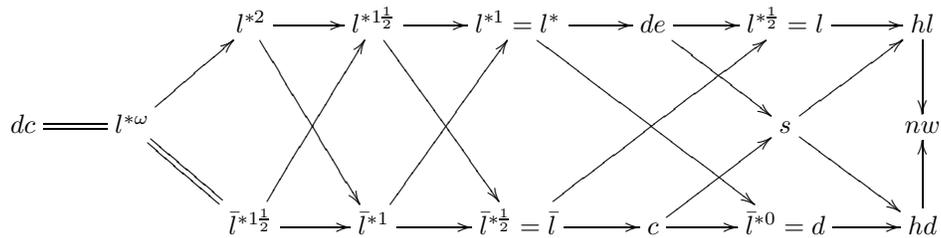
\xymatrix{
&&l^{*\kern-0.5pt 2}\ar[r]\ar[rdd]&l^{*1\kern-1pt\frac12}\ar[r]\ar[rdd]&l^{*\kern-0.5pt 1}=l^*\ar[r]\ar[rrdd]&de\ar[r]\ar[rd]&l^{*\kern-1pt\frac12}=l\ar[r]&hl\ar[d]\\
\dc\ar@{=}[r]&l^{*\w}\ar@{=}[rd]\ar[ru]&&&&&s\ar[ru]\ar[rd]&nw\\
&&\bar l^{*1\kern-1pt\frac12}\ar[r]\ar[ruu]&\bar l^{*1}\ar[r]\ar[ruu]&\bar l^{*\kern-1pt\frac12}=\bar l\ar[r]\ar[rruu]&c\ar[r]\ar[ru]&\bar l^{*0}=d\ar[r]&hd\ar[u]
}
$$
\smallskip

\begin{remark} Some inequalities in this diagram can be strict. In particular, there exist:
\begin{itemize}
\item a normal space $X$ (for example, the ordinal segment $[0,\w_1)$~) with $e(X)=de(X)=\w<c(X)=l(X)$;
\item a compact Hausdorff space $X$ (for example, $[0,\w_1]$) such that $\bar l(X)=l(X)=\w<c(X)$;
\item a normal space $X$ with $l^{*1\kern-1pt\frac12}(X)=\w<l^{*1}(X)$ (see \cite{Jun});
\item a Moore space $X$ with $l^*(X)=d(X)=\w<e(X)=de(X)$ (see \cite[3.2.3.1]{DRRT});
\item a Moore space $X$ with $c(X)=\w<l^*(X)$ (see \cite[3.2.3.2]{DRRT});
\item a Moore space $X$ with $l^{*1\kern-1pt\frac12}(X)=\w<c(X)$ (see \cite[3.2.3.3]{DRRT});
\item a Moore locally separable space $X$  with $l^{*2}(X)=\bar l^{*1}(X)=\w<l^{*1\kern-1pt\frac12}(X)$ (see \cite[3.2.3.4]{DRRT});
\item a Moore space $X$ with $l^{\w}(X)=\bar l^{*1\kern-1pt\frac12}(X)=\w<l^{*2}(X)\le \bar l^{*1}(X)$ (see \cite[3.2.3.5]{DRRT});
\item a Hausdorff space $X_n$ with $l^{*n\kern-1pt\frac12}(X_n)=\w<l^{*n}(X_n)$ for every $n\in\IN$ (see \cite[3.3.1]{DRRT}).
\end{itemize} There are consistent examples of normal spaces $X$ with $l^{*1}(X)=d(X)=\w$ and $e(X)=\mathfrak c$, see \cite{LM}. On the other hand, it is not known (see \cite{BM} or \cite{LM}) if there is a ZFC-example of a normal space $X$ with $l^*(X)=\w<e(X)$.
\end{remark}

\begin{problem} Construct an example of a (Tychonoff) topological space $X$ with $\bar l^{*1\kern-1pt\frac12}(X)<\bar l^{*1}(X)$.
\end{problem}

\section{Pre-uniformities and quasi-uniformities}\label{s2}

In this section we collect the necessary preliminary information on entourages, pre-uniformities, and operations on them.

\subsection{Binary words}

In this subsection we consider some structures on the set $\{+,-\}^{<\w}=\bigcup_{n\in\w}\{+,-\}^n$ of binary words in the alphabet $\{+,-\}$. This set is a semigroup with respect to the operation of concatenation of words. The empty word is the unit of this semigroup, so $\{+,-\}^{<\w}$ is a monoid.
In fact, it is a free monoid over the set $\{+,-\}$. This monoid carries a natural partial order: for two words $v\in \{+,-\}^n\subset\{+,-\}^{<\w}$ and $w\in\{+,-\}^m\subset\{+,-\}^{<\w}$ we write $v\le w$ if there is an injective map $i:n\to m$ such that $v=w\circ i$. So, $v\le w$ if the word $v$ can be obtained from $w$ by deleting some letters.

For every $n\in\w$ define the {\em alternation words} $\pm n$ and $\mp n$ by recursion: put $\mp 0=\pm 0$ be the empty word and $\pm (n+1)=+(\mp n)$, $\mp(n+1)=-(\pm n)$ for $n\in\w$.

\subsection{Entourages}

By an {\em entourage} on a set $X$ we understand any subset $U\subset X\times X$ containing the diagonal $\Delta_X=\{(x,y)\in X\times X:x=y\}$ of $X\times X$. For an entourage $U$ on $X$, point $x\in X$ and  subset $A\subset X$ let $B(x;U)=\{y\in X:(x,y)\in U\}$ be the {\em $U$-ball} centered at $x$, and $B(A;U)=\bigcup_{a\in A}B(a;U)$ be the {\em $U$-neighborhood} of $A$ in $X$. Since $U=\bigcup_{x\in X}\{x\}\times B(x;U)$, the entourage $U$ can be recovered from the family of $U$-balls $\{B(x;U):x\in X\}$.

Now we define some operations on entourages.
For two entourages $U,V$ on $X$ let $$U^{-1}=\{(x,y)\in X\times X:(y,x)\in U\}$$ be the {\em inverse} entourage and $$UV=\{(x,z)\in X\times X:\exists y\in X\;\mbox{ such that $(x,y)\in U$ and $(y,z)\in V$}\}$$be the {\em composition} of $U$ and $V$. It is easy to see that $(UV)^{-1}=V^{-1}U^{-1}$. For every entourage $U$ on $X$ define its powers $U^n$, $n\in\IZ$, by the formula: $U^0=\Delta_X$ and $U^{n+1}=U^nU$, $U^{-n-1}=U^{-n}U^{-1}$ for $n\in\w$.

The powers $U^n$, $U^{-n}$, $n\in\IN$, of an entourage $U$ are particular cases of the verbal powers $U^v$ where $v\in \{+,-\}^{<\w}:=\bigcup_{n\in\w}\{+,-\}^n$ is a word in the two-symbol alphabet $\{+,-\}$. The verbal powers $U^v$ of an entourage $U$ on $X$ are defined by induction: for the empty word $v=\emptyset\in\{+,-\}^0$ we put $U^\emptyset=\Delta_X$ and for a word $v\in\{+,-\}^{<\w}$ we define $$U^{v+}:=U^vU,\;\;U^{+v}:=UU^v,\;\;U^{v-}=U^vU^{-1},\;\;U^{-v}:=U^{-1}U^v.$$

The following two properties of verbal products follow  immediately from the definitions:

\begin{lemma}\label{l2.1} Let $v,w\in\{+,-\}^{<\w}$ be two words in the alphabet $\{+,-\}$. Then
\begin{enumerate}
\item $U^{vw}=U^vU^{w}$.
\item If $v\le w$, then $U^v\subset U^w$.
\end{enumerate}
\end{lemma}

For alternating words $\pm n$ and $\mp n$ the verbal powers $U^{\pm n}$ and $U^{\mp n}$ are called the {\em alternating powers} of $U$.

The following lemma shows that the alternating power $U^{\mp2}$ of an entourage $U$ is equivalent to taking the star with respect to the cover $\U=\{B(x;U):x\in X\}$.

\begin{lemma}\label{l2.2} For any entourage $U$ on a set $X$ and a point $x\in X$ we get $B(x;U^{-1}U)=\St(x;\U)$ where $\U=\{B(x;U):x\in X\}$. Consequently, $B(x;U^{\mp 2n})=B(x;(U^{-1}U)^n)=\St^n(x;\U)$ for every $n\in\IN$.
\end{lemma}

\begin{proof} Observe that for any $x\in X$ and $y\in B(x;U^{-1}U)$, there is a point $z\in X$ such that $(x,z)\in U^{-1}$ and $(z,y)\in U$, which implies $x,y\in B(z;U)$ and hence $y\in\St(x;\U)$. So, $B(x;U^{-1}U)\subset \St(x;\U)$.

Now assume that $y\in\St(x;\U)$. Then $y,x\in B(z;U)$ for some $z\in X$ and hence $(x,z)\in U^{-1}$, $(z,y)\in U$, which implies $(x,y)\in U^{-1}U$. So, $\St(x;\U)\subset B(x;U^{-1}U)$.
\end{proof}

\subsection{Uniformities, quasi-uniformities, and pre-uniformities}

A family $\U$ of entourages on a set $X$ is called a {\em uniformity} on $X$ if it satisfies the following four axioms:
\begin{itemize}
\item[(U1)] for any $U\in\U$, every entourage $V\subset X\times X$ containing $U$ belongs to $\U$;
\item[(U2)] for any entourages $U,V\in\U$ there is an entourage $W\in\U$ such that $W\subset U\cap V$;
\item[(U3)] for any entourage $U\in\U$ there is an entourage $V\in\U$ such that $VV\subset U$;
\item[(U4)] for any entourage $U\in\U$ there is an entourage $V\in\U$ such that $V\subset U^{-1}$.
\end{itemize}
A family $\U$ of entourages on $X$ is called a {\em quasi-uniformity} (resp. {\em pre-uniformity\/}) on $X$ if it satisfies the axioms (U1)--(U3) (resp. (U1)--(U2)~).
So, each uniformity is a quasi-uniformity and each quasi-uniformity is a pre-uniformity. Observe that a pre-uniformity is just a filter of entourages on $X$.

A subfamily $\mathcal B\subset\U$ is called a {\em base} of a pre-uniformity $\U$ on $X$ if each entourage $U\in\U$ contains some entourage $B\in\mathcal B$. Each base of a pre-uniformity satisfies the axiom (U2). Conversely, each family  $\mathcal B$ of entourages on $X$ satisfying the axiom $(U2)$ is a base of a unique pre-uniformity $\langle\mathcal B\rangle$ consisting of entourages $U\subset X\times X$ containing some entourage $B\in\mathcal B$. If the base $\mathcal B$ satisfies the axiom (U3) (and (U4)), then the pre-uniformity $\langle\mathcal B\rangle$ is a quasi-uniformity (and a uniformity).

Next we define some operations over pre-uniformities.
Given two pre-uniformities $\U,\V$ on a set $X$ put $$\U^{-1}=\{U^{-1}:U\in\U\},\;\;\U\vee\V=\{U\cap V:U\in\U,\;V\in\V\},\;\;\U\wedge\V=\{U\cup V:U\in\U,\;V\in\V\}$$ and let $\U\V$ be the pre-uniformity generated by the base $\{UV:U\in\U,\;V\in\V\}$. Observe that the family of pre-uniformities on a set $X$ endowed with the inclusion order is a (complete) lattice with $\U\vee\V$ and $\U\wedge\V$ being the operations of supremum and infimum in this lattice.

For any pre-uniformity $\U$ on a set $X$ and a binary word $v\in\{+,-\}^{<\w}$ let $\U^v$ be the pre-uniformity on $X$ generated by the base $\{U^v:U\in\U\}$.

Observe that a pre-uniformity $\U$ on a set is a quasi-uniformity (and a uniformity) if and only if
$\U\U=\U$ (and $\U^{-1}=\U$). This fact implies:

\begin{proposition} For any non-empty binary word $v\in\{+,-\}^{<\w}$ and any quasi-uniformity $\U$ the pre-uniformity $\U^v$ is equal to $\U^{\pm n}$ or $\U^{\mp  n}$ for some $n\in\w$. If $\U$ is a uniformity, then $\U^v=\U$.
\end{proposition}

For any pre-uniformity $\U$ and every $n\in\IN$ consider another two pre-uniformities: $$\U^{\wedge n}:=\U^{\pm n}\wedge\U^{\mp n}\mbox{ \ and \ }\U^{\vee n}:=\U^{\pm n}\vee\U^{\mp n}.$$ Observe that these pre-uniformities fit into the following diagram (in which an arrow $\V\to\W$ indicates that $\V\subset\W$):
$$
\xymatrix{
&&\U^{\pm n}\ar[rd]\\
\U^{\vee(n+1)}\ar[r]&\U^{\wedge n}\ar[ru]\ar[rd]&&\U^{\vee n}\ar[r]&\U^{\wedge(n-1)}\\
&&\U^{\mp n}\ar[ru]
}
$$

\begin{proposition} Let $\U$ be a quasi-uniformity on a set $X$. If  $\U^{\pm n}=\U^{\mp n}$ for some $n\in\IN$, then the pre-uniformity  $\U^{\pm n}=\U^{\mp n}$ is a uniformity and $\U^{\pm n}=\U^{\pm m}=\U^{\mp m}=\U^{\vee m}=\U^{\wedge m}$ for every $m\ge n$.
\end{proposition}

\begin{proof} Assume that $\U^{\pm n}=\U^{\mp n}$ for some $n\in\IN$. Taking into account that $(\U^{\pm n})^{-1}\in\{\U^{\pm n},\U^{\mp n}\}=\{\U^{\pm n}\}$, we conclude that $(\U^{\pm n})^{-1}=\U^{\pm n}$.
Next, we prove that $\U^{\pm n}\subset\U^{\pm m}$ for every $m\ge n$. This will be done by induction on $m\ge n$. For $m=n$ the inclusion $\U^{\pm n}\subset\U^{\pm m}$ is trivial. Assume that for some $m>n$ we have proved that $\U^{\pm n}\subset\U^{\pm (m-1)}$. The inclusion $\U^{\pm n}\subset\U^{\pm m}$ will be proved as soon as for every entourage $U\in\U$ we find an entourage $V\in\U$ such that
$V^{\pm m}\subset U^{\pm n}$. Since $\U^{\pm n}\subset\U^{\mp n}$, for the entourage $U\in\U$ we can find an entourage $\widetilde U\in\U$ such that $\widetilde U^{\mp n}\subset U^{\pm n}$. Since $\U$ is a quasi-uniformity, we can additionally assume that $\widetilde U^2\subset U$.

Since $\U^{\mp n}=\U^{\pm n}\subset \U^{\pm (m-1)}$, there is an entourage $V\in\U$, $V\subset\widetilde U$, such that $V^{\pm (m-1)}\subset \widetilde U^{\mp n}\cap\widetilde U^{\pm n}$.

If $m$ is odd, then $V^{\mp m}=V^{-1}V^{\pm (m-1)}\subset \widetilde U^{-1}\widetilde U^{\mp n}\subset U^{\mp n}$ and after inversion, $V^{\pm m}\subset (\widetilde U^{\mp n})^{-1}$. If $n$ is odd, then  $V^{\pm m}\subset (\widetilde U^{\mp n})^{-1}=\widetilde U^{\pm n}\subset U^{\pm n}$. If $n$ is even, then  $V^{\pm m}\subset (\widetilde U^{\mp n})^{-1}=\widetilde U^{\mp n}\subset U^{\pm n}$ by the choice of $\widetilde U$.

If $m$ is even, then $V^{\pm m}=V^{\pm (m-1)}V^{-1}\subset (\widetilde U^{\mp n}\cap\widetilde U^{\pm n})\tilde U^{-1}$. If $n$ is odd, then $V^{\pm m}\subset \widetilde U^{\mp n}\widetilde U^{-1}\subset U^{\mp n}$ and after inversion we get $V^{\pm m}=(V^{\pm m})^{-1}\subset (U^{\mp n})^{-1}=U^{\pm n}$. If $n$ is even, then $V^{\pm m}\subset \tilde U^{\pm n}\widetilde U^{-1}\subset U^{\pm n}$ (which follows from $\tilde U^{-2}\subset U^{-1}$).

Therefore, $\U^{\pm n}\subset\U^{\pm m}$ for all $m\ge n$. In particular, $\U^{\pm n}\subset \U^{\pm 2n}\subset \U^{\pm n}\U^{\pm n}$, which means that $\U^{\pm n}$ is a quasi-uniformity. Since $(\U^{\pm n})^{-1}=\U^{\pm}$, the quasi-uniformity $\U^{\pm n}$ is a uniformity.
\end{proof}

\subsection{Boundedness numbers of pre-uniform spaces}
For any pre-uniformity $\U$ on a set $X$ define two cardinal characteristics:
\begin{itemize}
\item the {\em boundedness number} $\ell(\U)$, defined as the smallest cardinal $\kappa$ such that for any entourage $U\in\U$ there is a subset $A\subset X$ of cardinality $|A|\le \kappa$ such that $X=B(A;U)$;
\item the {\em sharp boundedness number} $L(\U)$, defined as the smallest cardinal $\kappa$ such that for any entourage $U\in\U$ there is a subset $A\subset X$ of cardinality $|A|<\kappa$ such that $X=B(A;U)$.
\end{itemize}
Since $\ell(\U)=L(\U)_-$, the sharp boundedness number $L(\U)$ determines the value of the boundedness number $\ell(\U)$.

Each pre-uniformity $\U$ on a set $X$ generates a topology $\tau_\U$ consisting of all subsets $W\subset X$ such that for each point $x\in W$ there is an entourage $U\in\U$ with $B(x;U)\subset W$, see \cite{CGKKM}. This topology $\tau_\U$ will be referred to as {\em the topology generated by the pre-uniformity} $\U$. If $\U$ is a quasi-uniformity, then for each point $x\in X$ the family of balls $\{B(x;U):U\in\U\}$ is a neighborhood base of the topology $\tau_\U$ at $x$. This implies that for a quasi-uniformity $\U$ on a set $X$ the topology $\tau_\U$ is Hausdorff if and only if  for any distinct points $x,y\in X$ there is an entourage $U\in\U$ such that $B(x;U)\cap B(y;U)=\emptyset$ if and only if $\bigcap\U\U^{-1}=\Delta_X$.
It is known (see \cite{Ku1} or \cite{Ku2}) that the topology of each topological space $X$ is generated by a suitable quasi-uniformity (in particular, the Pervin quasi-uniformity, generated by the subbase consisting of the entourages $(U\times U)\cup \big((X\setminus U)\times X\big)$ where $U$ runs over open sets in $X$).

For a pre-uniformity $\U$ on a topological space $X$ consider another two cardinal characteristics:
\begin{itemize}
\item the {\em dense boundedness number} $\bar \ell(\U)$, defined as the smallest cardinal $\kappa$ such that for any entourage $U\in\U$ there is a subset $A\subset X$ of cardinality $|A|\le \kappa$ such that $B(A;U)$ in dense in $X$;
\item the {\em dense sharp boundedness number} $\bar L(\U)$, defined as the smallest cardinal $\kappa$ such that for any entourage $U\in\U$ there is a subset $A\subset X$ of cardinality $|A|<\kappa$ such that $B(A;U)$ is dense in $X$.
\end{itemize}
The following diagram describes the interplay between the cardinal characteristics $\ell(\U)$, $L(\U)$, $\bar\ell(\U)$, and $\bar L(\U)$.
$$
\xymatrix{
\ell(\U)\ar@{=}[r]&L(\U)_-\ar[r]&L(\U)\\
\bar \ell(\U)\ar@{=}[r]\ar[u]&\bar L(\U)_-\ar[r]&\bar L(\U).\ar[u]
}
$$

Observe that for any pre-uniformities $\U\subset\V$ on a topological space we get $\ell(\U)\le\ell(\V)$, $\bar\ell(\U)\le\bar\ell(\V)$, $L(\U)\le L(\V)$, and $\bar L(\U)\le\bar L(\V)$. For any binary words $v\le w$ in $\{+,-\}^{<\w}$ we get $\U^w\subset\U^v$ and hence
$$\ell(\U^w)\le\ell(\U^v),\;\;\bar\ell(\U^w)\le\bar\ell(\U^v),\;\;L(\U^w)\le L(\U^v),\;\;\mbox{ and }\;\;\bar L(\U^w)\le \bar L(\U^v).$$

For a binary word $v\in\{+,-\}^{<\w}$ and a quasi-uniformity $\U$ we put $\ell^v(\U):=\ell(\U^v)$. In particular, for every $n\in\IN$ we put $\ell^{\pm n}(\U):=\ell(\U^{\pm n})$, $\ell^{\mp n}(\U):=\ell(\U^{\mp n})$, $\ell^{\wedge n}(\U):=\ell(\U^{\wedge n})$, and $\ell^{\vee n}(\U):=\ell(\U^{\vee n})$. Taking into account the inclusion relations between the pre-uniformities $\U^{\pm n}$, $\U^{\mp n}$, $\U^{\wedge n}$ and $\U^{\vee n}$, we get the following diagram.
$$
\xymatrix{
&&\ell^{\pm n}(\U)\ar[rd]\\
\ell^{\vee(n+1)}(\U)\ar[r]&\ell^{\wedge n}(\U)\ar[ru]\ar[rd]&&\ell^{\vee n}(\U)\ar[r]&\ell^{\wedge(n-1)}(\U)\\
&&\ell^{\mp n}(\U)\ar[ru]
}
$$

\section{Universal (pre- and quasi-) uniformities on topological spaces}\label{s3n}

Let $X$ be a topological space. An entourage $U$ on $X$ is called a {\em neighborhood assignment} if for every $x\in X$ the $U$-ball $B(x;U)$ is a neighborhood of $x$. The family $\pU_X$ of all neighborhood assignments on a topological space $X$ is a pre-uniformity called the {\em universal pre-uniformity} on $X$. It contains any pre-uniformity generating the topology of $X$ and is equal to the union of all pre-uniformities generating the topology of $X$.

The universal pre-uniformity $\pU_X$ contains
\begin{itemize}
\item the {\em universal quasi-uniformity} $\qU_X=\bigcup\{\U\subset\pU_X:\U$ is a quasi-uniformity on $X\}$, and
\item the {\em universal uniformity} $\U_X=\bigcup\{\U\subset\pU_X:\U$ is a uniformity on $X\}$
\end{itemize}
of $X$.
It is clear that $\U_X\subset q\U_X\subset  \pU_X$. Moreover, $q\U_X\subset \pU_X^{n+1}\subset \pU^n_X\subset \pU_X$ for every $n\in\IN$.

Since the topology of any topological space $X$ is generated by a quasi-uniformity, the universal quasi-uniformity $\qU_X$ generates the topology of $X$. In contrast, the universal uniformity $\U_X$ generates the topology of $X$ if and only if the space $X$ is completely regular.

Therefore, each topological space $X$ carries many canonical pre-uniformities: $\pU_X$, $\qU_X$, $\U_X$, their verbal powers, and the Boolean operations over their verbal powers.
The following diagram describe the inclusion relation between these canonical pre-uniformities. In this diagram for two pre-uniformities $\V,\W$ an arrow $\V\to\W$ indicates that $\V\subset\W$.
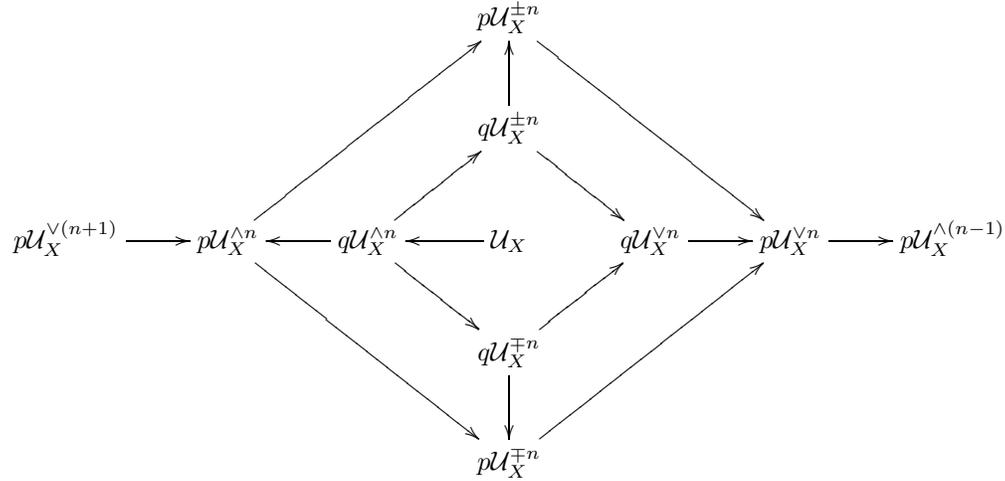
$$
\xymatrix{
&&&\pU_X^{\pm n}\ar[rrdd]\\
&&&\qU_X^{\pm n}\ar[rd]\ar[u]\\
\pU_X^{\vee(n+1)}\ar[r]&\pU_X^{\wedge n}\ar[rruu]\ar[rrdd]&\qU_X^{\wedge n}\ar[l]\ar[ru]\ar[rd]&\U_X\ar[l]&\qU_X^{\vee n}\ar[r]&\pU_X^{\vee n}\ar[r]&\pU_X^{\wedge (n-1)}\\
&&&\qU_X^{\mp n}\ar[ru]\ar[d]\\
&&&\pU_X^{\mp n}\ar[rruu]
}
$$

Now we detect spaces for which one of the inclusion $\U_X\subset \qU_X\subset \pU_X$ turns into equality.

\begin{proposition} A topological space $X$ has $\U_X=\qU_X$ if and only if $X$ is discrete.
\end{proposition}

\begin{proof} The ``if'' part is trivial. To prove the ``only if'' part, assume that the space $X$ is not discrete. Fix a non-isolated point $x_0\in X$ and consider the neighborhood assignment $U=\big(\{x_0\}\times X\big)\cup (X\setminus\{x_0\})^2$. Since $UU=U$, the entourage $U$ belongs to the universal quasi-uniformity $\qU_X$. Assuming that $\qU_X=\U_X$, we could find a neighborhood assignment $V\in\U_X$ such that $V^{-1}=V\subset U$. Since the point $x_0$ is not isolated, the ball $B(x_0;V)$ contains some point $x\ne x_0$. Then $X\setminus \{x_0\}=B(x;U)\supset B(x;V^{-1})\in x_0$, which is a contradiction.
\end{proof}

On the other hand, we have:

\begin{proposition}\label{p3.2n} Any paracompact space $X$ has $\U_X=\pU_X^{\mp 2}$.
\end{proposition}

\begin{proof} Given any neighborhood assignment $U\in\pU_X$, we need to find an entourage $V\in\U_X$ such that $V\subset U^{-1}U$. By Lemma~\ref{l2.2}, for every $x\in X$ we get $B(x;U^{-1}U)=\St(x;\U)$ where $\U=\{B(x;U):x\in X\}$. By the paracompactness of $X$ we can construct a sequence of open covers $(\U_n)_{n\in\w}$ of $X$ such that $\U_0=\U$ and for every $n\in\IN$ the cover $\St(\U_n)=\{\St(U,\U_{n}):U\in\U_n\}$ refines the cover $\U_{n-1}$. Using the sequence $(\U_n)_{n\in\w}$, we can show that the neighborhood assignment $V=\bigcup_{x\in X}\big(\{x\}\times\St(x;\U_n)\big)=U^{-1}U$ belongs to the universal uniformity $\U_X$ on $X$.
\end{proof}

The interplay between the pre-uniformities $\qU_X$ and $\pU_X$ is even more interesting.

\begin{proposition} Each countable $T_1$-space $X$ has $\pU_X=\qU_X$.
\end{proposition}

\begin{proof} If $X$ is finite, then the equality $\pU_X=\qU_X$ is trivial. So, we assume that $X$ is infinite. Given any neighborhood assignment $U\in\pU_X$ we shall construct a neighborhood assignment $V\in\pU_X$ such that $V=VV\subset U$. Let $X=\{x_n\}_{n\in\w}$ be an injective enumeration of $X$. By induction for every $n\in\w$ choose an open neighborhood $O(x_n)\subset X$ of $x_n$ such that $O(x_n)\subset B(x_n;U)\cap \bigcap\{O(x_k):k<n,\;x_n\in O(x_k)\}\setminus\{x_k\}_{k<n}$. We claim that the neighborhood assignment $V=\bigcup_{k\in\w}\{x_k\}\times O(x_k)$ has the desired property: $VV=V\subset U$. The inclusion   $V\subset U$ follows from $B(x_n;V)=O(x_n)\subset B(x_n;U)$ for $n\in\w$. To see that $VV\subset V$, fix any pair $(x_n,x_m)\in VV$ and find a point $x_k\in X$ with $(x_n,x_k),(x_k,x_m)\in V$.
It follow from $x_k\in B(x_n;V)=O(x_n)$ that $k\ge n$. By analogy, $x_m\in B(x_k;V)=O(x_k)$ implies $m\ge k$. By the choice of $O(x_k)$, the inclusion $x_k\in O(x_n)$ implies $x_m\in O(x_k)\subset O(x_n)= B(x_n;V)$ and hence $(x_n,x_m)\in V$.
The equality $VV=V$ implies that $V\in\qU_X$.
\end{proof}

Also the equality $\pU_X=\qU_X$ holds for hereditarily paracompact scattered spaces.
We recall that a topological space $X$ is {\em hereditarily paracompact} if each subspace of $X$ is paracompact.
A topological space $X$ is {\em scattered} if each subspace $A\subset X$ has an isolated point. By a result of Telgarsky \cite{Telga}, each scattered paracompact space $X$ is {\em strongly zero-dimensional} in the sense that each open cover of $X$ can be refined by a disjoint open cover.

\begin{proposition}\label{p3.4n} Each scattered hereditarily paracompact space $X$ has $\pU_X=\qU_X$.
\end{proposition}

\begin{proof} Given any neighborhood assignment $U\in\pU_X$ we shall find a neighborhood assignment  $V\in\qU_X$ such that $V=VV\subset U$. Consider the family $\W$ of open sets $W\subset X$ for which there exists an neighborhood assignment $V_W\in\pU_W$ such that $V_W=V_WV_W\subset U\cap(W\times W)$.
Observe that for each isolated point $x\in X$ the singleton $\{x\}$ belongs to $\W$, which implies that the union $\bigcup\W$ is a dense open subset of $X$. We claim that $\bigcup\W=X$. Assuming that the remainder $A=X\setminus \bigcup\W$ is not empty, we could find an isolated point $x\in A$.
Since the set $A\setminus\{x\}$ is closed in $A$, the set $\dot W=\{x\}\cup\bigcup\W$ is open in $X$.
To derive a contradiction, we shall prove that $\dot W\in\W$.

Since $X$ is (strongly) zero-dimensional, the point $x$ has a closed-and-open neighborhood $O_x\subset \dot W$ such that $O_x\subset B(x;U)$ and $O_x\cap A=\{x\}$. Since the space $X$ is hereditarily paracompact, its open subset $\bigcup\W$ is paracompact and, being scattered, is strongly zero-dimensional according to \cite{Telga}. Consequently, we can find a disjoint open cover $\W'$ which refines the open cover $\W\wedge\{O_x,X\setminus O_x\}=\{W\cap O_x,W\setminus O_x:W\in\W\}$ of the set $\bigcup\W$. By definition of $\W$, for every set $W\in\W'$ there is an entourage $V_{W'}\in\pU_{W'}$ such that $V_{W'}=V^2_{W'}\subset U\cap(W'\times W')$. Then the neighborhood assignment $V=(\{x\}\times O_x)\cup\bigcup_{W'\in\W'}V_{W'}$ on $\dot W$ has the desired property: $V=VV\subset U$, witnessing that $\dot W\in\W$. But this is not possible as $\dot W\not\subset\bigcup\W$. This contradiction shows that $\bigcup\W=X$. Using the strong zero-dimensionality of $X$ and repeating the above argument, we can find a neighborhood assignment $V$ on $X$ such that $V=VV\subset U$. The equality $V=VV$ implies that $V\in\qU_X$ and hence $U\in\qU_X$.
\end{proof}

The (hereditary) paracompactness of the scattered space $X$ in Proposition~\ref{p3.4n} is essential as shown by the following example.

\begin{example} The ordinal $X=\w_1=[0,\w_1{[}$ endowed with the order topology has $\pU_X\ne \pU_X^2$ and hence $\pU_X\ne \qU_X$. On the other hand, $\U_X=\pU_X^{\mp2}$.
\end{example}

\begin{proof} Let us recall that a subset $S\subset \w_1$ is called {\em stationary} if $S$ meets each closed unbounded subset of $\w_1$. By \cite[23.4]{JW2}, the space $\w_1$ contains a disjoint family $\{S_\alpha\}_{\alpha<\w_1}$ consisting of $\w_1$ many stationary sets. We lose no generality assuming that each stationary set $S_\alpha$ is contained in the order interval ${]}\alpha,\w_1{[}$. Consider a neighborhood assignment $U$ on $X$ such that for any ordinal $\alpha<\w_1$ and point $x\in S_\alpha$ we get $B(x;U)={]}\alpha,x]$. We claim that $U\notin\pU_X^2$. Assuming the converse, we could find a neighborhood assignment $V\in\pU_X$ such that $VV\subset U$. For every ordinal $\alpha\in X=\w_1$ find an ordinal $f(\alpha)<\alpha$ such that ${]}f(\alpha),\alpha]\subset B(\alpha;V)$. By Fodor's Lemma~\cite[21.12]{JW2}, for some stationary set $S\subset\w_1$ the restriction $f|S$ is constant and hence $f(S)=\{c\}$ for some ordinal $c$.
We lose no generality assuming that $s>c$ for any ordinal $s\in S$.

Take any ordinal $\alpha>c$. The set $S_\alpha$, being stationary, meets the closed unbounded set
$\overline{S}$. Then we can find a point $x\in S_\alpha\cap\overline{S}$ and a point $s\in S\cap B(x;V)$. Then $$c+1\in{]}c,s]={]}f(s),s]\subset B(s;V)\subset B(x;VV)\subset B(x;U)={]}\alpha,x]$$ which is not possible as $c<\alpha$. This contradiction completes the proof of the inequality $\pU_X\ne\pU_X^2$.
\smallskip

Now we prove that $\U_X=\pU_X^{\mp2}$. Given any neighborhood assignment $U\in\pU_X$, we need to show  that $U^{-1}U\in\U_X$. For any ordinal $\alpha\in \w_1$ find an ordinal $f(\alpha)<\alpha$ such that ${]}f(\alpha),\alpha]\subset B(\alpha;U)$. By Fodor's Lemma~\cite[21.12]{JW2}, for some stationary set $S\subset [0,\w_1{[}$ the restriction $f|S$ is constant and hence $f(S)=\{\beta\}$ for some countable ordinal $\beta$. Then for every countable ordinal $x>\beta$ we can find an ordinal $\alpha\in S$ with $\alpha>x$ and conclude that $x\in {]}\beta,\alpha]\subset B(\alpha;U)$, which implies $\alpha\in B(x;U^{-1})$ and ${]}\beta,\alpha]\subset B(\alpha;U)\subset B(x;U^{-1}U)$. Consequently, ${]}\beta,\w_1{[}=\bigcup_{S\ni\alpha>x}{]}\beta,\alpha]\subset B(x;U^{-1}U)$. By Proposition~\ref{p3.2n}, for the compact metrizable space $[0,\beta]$ there is a sequence $(V_n)_{n\in\w}$ of neighborhood assignments such that $V_0\subset U^{-1}U\cap[0,\beta]^2$ and $V_{n+1}^2\subset V_n=V_n^{-1}\subset [0,\beta]^2$ for every $n\in\w$. For every $n\in\w$ put $W_n=V_n\cup\,{]}\beta,\w_1{[}^2$ and observe that $W_0\subset U^{-1}U$ and $(W_n)_{n\in\w}$ is a sequence of neighborhood assignments on $X$ such that $W_{n+1}^2\subset W_n=W_n^{-1}$ for every $n\in\w$. This implies that $\{W_n\}_{n\in\w}\subset\pU_X$ is a base of a uniformity on $X$ and hence the entourages $W_0$ and $U^{-1}U$ belong to the universal uniformity $\U_X$ of $X$.
\end{proof}

Now we shall characterize the metrizable spaces $X$ with $\pU_X=\qU_X$. Let us recall that a topological space $X$ is called a {\em $Q$-set} if each subset in $X$ is of type $F_\sigma$ (see \cite[\S4]{Miller},  \cite{BMZ} for more information on $Q$-sets). We define a topological space $X$ to be a {\em $Q_\w$-set} if for any increasing sequence $(X_n)_{n\in\w}$ of sets with $X=\bigcup_{n\in\w}X_n$ there exists a sequence $(F_n)_{n\in\w}$ of closed subsets in $X$ such that $X=\bigcup_{n\in\w}F_n$ and $F_n\subset X_n$ for all $n\in\w$. It is easy to see that each $Q$-set is a $Q_\omega$-set.
On the other hand, each metrizable $Q_\w$-set is perfectly meager.

A topological space $X$ is called {\em perfectly meager} if each crowded subspace of $X$ is meager in itself. A topological space is {\em crowded\/} if it has no isolated points. More information of perfectly meager spaces can be found in \cite[\S5]{Miller}.

\begin{proposition}\label{p3.6n} Each (metrizable) $Q_\w$-set $X$ has cardinality $|X|<\w\cdot 2^{w(X)}$ (and is perfectly meager).
\end{proposition}

\begin{proof} Assume that an infinite topological space $X$ is a $Q_\w$-set.
\smallskip

First we prove that $|X|<\w\cdot2^{w(X)}$. To derive a contradiction, assume that $|X|\ge \w\cdot 2^{w(X)}$. Let $\tau$ denote the topology of the space $X$ and observe that $|\tau|\le 2^{w(X)}$. Denote by $\F$ the family of all closed subsets $F\subset X$ of cardinality $|F|\ge \w\cdot 2^{w(X)}$ and observe that $|\F|\le|\tau|\le 2^{w(X)}$. Let $\F=\{F_\alpha\}_{\alpha<|\F|}$ be an enumeration of the set $\F$.

\begin{claim} For every $\alpha<|\F|$ there is an injective map $i_\alpha:\w\to F_\alpha$ such that the indexed family $\big(i_\alpha(\w)\big)_{\alpha<|\F|}$ is disjoint.
\end{claim}

\begin{proof} If $w(X)$ is finite, then so is the family $\F$. By induction for every $n\in\w$ choose an injective map $j_n:\F\to X$ such that $j_n(F)\in F\setminus \bigcup_{k<n}j_k(\F)$ for every $F\in\F$. For every $\alpha<|F|$ put $i_\alpha(n)=j_n(F_\alpha)$ and observe that the map $i_\alpha:\w\to F_\alpha$ is injective and the indexed family $\big(i_\alpha(\w)\big)_{\alpha<|\F|}$ is disjoint.

If $w(X)$ is infinite, then $|\F|\ge\w\cdot 2^{w(X)}\ge 2^\w$. In this case by transfinite induction, for every ordinal $\alpha<|\F|$ fix an injective map $i_\alpha:\w\to F_\alpha\setminus \bigcup_{\beta<\alpha}i_\beta(\w)$. The choice of $i_\alpha$ is always possible since $|\bigcup_{\beta<\alpha}i_\beta(\w)|=|\w\times \alpha|<|\w\times \F|\le \w\cdot 2^{w(X)}\le|F_\alpha|$. This construction ensures that the indexed family $\big(i_\alpha(\w)\big)_{\alpha<|\F|}$ is disjoint.
\end{proof}

For every $n\in\w$ consider the set
$X_n=X\setminus\{i_\alpha(m):m\ge n,\;\alpha<|\F|\}$ and observe that $X=\bigcup_{n\in\w}X_n$. Since $X$ is a $Q_\w$-set, there is an increasing sequence $(A_n)_{n\in\w}$ of closed subsets in $X$ such that $X=\bigcup_{n\in\w}A_n$ and $A_n\subset X_n$ for all $n\in\w$. The choice of the sets $X_n$, $n\in\w$, guarantees that $A_n\notin\F$  and hence $|A_n|<\w\cdot 2^{w(X)}$ for all $n\in\w$. If $w(X)$ is finite, then $\{A_n\}_{n\in\w}\subset \{X\setminus U:U\in\tau\}$ is a finite family of finite sets in $X$. It follows that the union $X=\bigcup_{n\in\w}A_n$ is finite, which contradicts the assumption $|X|\ge\w\cdot 2^{w(X)}$. So, we conclude that $w(X)$ is infinite. In this case K\"onig's Lemma (see Corollary 24 of \cite{JW1}) guarantees that  $\cf(2^{w(X)})>w(X)\ge \w$. Consequently, $|X|=\big|\bigcup_{n\in\w}A_n\big|<2^{w(X)}\le|X|$, which is a desired contradiction witnessing that $|X|<2^{w(X)}$.
\smallskip

Now assume that the space $X$ is metrizable. To prove that $X$ is perfectly meager, fix a crowded subspace $Z\subset X$. Using the well-known fact \cite[4.4.3]{Eng} that each metrizable space has a $\sigma$-discrete base, one can construct a countable disjoint family $(D_n)_{n\in\w}$ of dense sets in $Z$ such that $Z=\bigcup_{n\in\w}D_n$. Since $X$ is a $Q_\w$-set, for the increasing sequence $\big((X\setminus Z)\cup\bigcup_{k\le n}D_n\big)_{n\in\w}$ there exists an increasing sequence $(F_n)_{n\in\w}$ of closed subsets of $X$ such that $X=\bigcup_{n\in\w}F_n$ and $F_n\subset (X\setminus Z)\cup\bigcup_{k\le n}D_n$ for all $n\in\w$. For every $n\in\w$ the closed subset $F_n\cap Z$ is disjoint with the dense set $D_{n+1}$ in $Z$ and hence $F_n\cap Z$ is nowhere dense in $Z$. Since $Z=\bigcup_{n\in\w}F_n\cap Z$, the space $Z$ is meager.
\end{proof}

Thus for any metrizable space $X$ we have the implications:
$$\mbox{$Q$-set \ $\Ra$ $Q_\w$-set \ $\Ra$ perfectly meager}.$$
Now we can prove the promised characterization of metrizable spaces $X$ with $\pU_X=\qU_X$.

\begin{proposition}\label{p3.7n} For a metrizable space $X$ the following conditions are equivalent:
\begin{enumerate}
\item $\pU_X=\qU_X$;
\item $\pU_X=\pU_X^{\,2}$;
\item $X$ is a $Q_\w$-set.
\end{enumerate}
The equivalent conditions \textup{(1)--(3)} imply that $X$ is perfectly meager and $|X|<2^{w(X)}$.
\end{proposition}

\begin{proof} The implication $(1)\Ra(2)$ is trivial. To prove that $(2)\Ra(3)$ assume that $\pU_X=\pU_X^{\,2}$ and fix any increasing sequence $(X_n)_{n\in\w}$ of subsets such that $X=\bigcup_{n\in\w}X_n$. Denote by $X'$ the set of non-isolated points in $X$. For every point $x\in X$ let $n_x=\min\{n\in\w:x\in X_n\}$. Fix a metric  $d$ generating the topology of $X$. For a point $x\in X$ and $\e>0$ denote by $B_d(x;\e)=\{y\in X:d(x,y)<\e\}$ the open $\e$-ball centered at $x$. Choose a neighborhood assignment $U$ on $X$ such that $B(x;U)=\{x\}$ for any isolated point $x\in X$ and $B_d(x;2^{-n_x})\not\subset B(x;U)$ for any non-isolated point $x\in X$. Since $U\in\U_X=\pU_X^2$, there is a neighborhood assignment $V\in\pU_X$ such that $VV\subset U$. For every $n\in\w$ consider the set $Z_n=\{x\in X':B_d(x;2^{-n})\subset B(x;V)\}$ and its closure $\overline{Z}_n$ in $X'$. Observe that $X'=\bigcup_{n\in\w}Z_n=\bigcup_{n\in\w}\overline{Z}_n$. We claim that $\overline{Z}_n\subset X_n$ for every $n\in\w$. Given $n\in\w$ and point $x\in \overline{Z}_n$, find a point $z\in Z_n\cap B(x;V)\cap B_d(x;2^{-n-1})$. Then $B(x;2^{-n-1})\subset B(z;2^{-n})\subset B(z;V)\subset B(x;VV)\subset B(x;U)$, which implies that $n+1>n_x$ and hence $x\in X_n$. So, $\overline{Z}_n\subset X_n$. Write the open discrete subset $X\setminus X'$ as a countable union $X\setminus X'=\bigcup_{n\in\w}D_n$ of closed subsets $D_n$ of $X$. Then $X=\bigcup_{n\in\w}F_n$ for the increasing sequence $(F_n)$ of the closed sets $F_n=(D_n\cap X_n)\cup\overline{Z}_n\subset X_n$, $n\in\w$, which means that $X$ is a $Q_\w$-set.
\smallskip

$(3)\Ra(1)$ Assume that $X$ is a $Q_\w$-set. To prove that $\pU_X=\qU_X$, take any neighborhood assignment $U\in\pU_X$. Fix a metric $d\le 1$ generating the topology of the metrizable space $X$. For every $n\in\w$ consider the set $X_n=\{x\in X:B_d(x;3^{-n})\subset B(x;U)\}$. The space $X$, being a $Q_\w$-set, can be written as the union $X=\bigcup_{n\in\w}F_n$ of an increasing sequence $(F_n)_{n\in\w}$ of closed subsets of $X$ such that $F_n\subset X_n$ for every $n\in\w$. For every point $x\in X$ let $n_x=\min\{n\in\w:x\in F_n\}$. Observe that $B_d(x;3^{-n_x})\subset B(x;U)$.

For every number $k\in\w$ consider the neighborhood assignment $V_k$ on $X$ assigning to each point $x\in X$ the open ball $B(x;V_k)=B_d(x;3^{-k-1}\cdot\min\{3^{-n_x},d(x;F_{n_x-1})\})$. Here we assume that $F_{-1}=\emptyset$ and $d(x,\emptyset)=1$. It follows that $B(x;V_0)\subset B_d(x;3^{-n_x})\subset B(x;U)$ and hence $V_0\subset U$. The inclusion $V_0\in\qU_X$ will follow as soon as we check that $V_{k+1}^2\subset V_{k}$ for every $k\in\w$. Take any points $x,y,z\in X$ with $(x,y)\in V_{k+1}$ and $(y,z)\in V_{k+1}$. By the definition of the entourage $V_{k+1}$, the ball $B(x;V_{k+1})\ni y$ does not intersect the set $F_{n_x-1}$, which implies that $n_y\ge n_x$. By the same reason, $n_z\ge n_y$.
The inclusions $y\in B(x;V_{k+1})$ and $z\in B(y;V_{k+1})$ imply $d(x,y)\le 3^{-k-1}\cdot\min\{3^{-n_x},d(x,F_{n_x-1})\}\le \frac13d(x,F_{n_x-1})$ and $d(y,z)\le 3^{-k-1}\cdot\min\{3^{-n_y},d(y,F_{n_y-1})\}$.
It follows that $$d(y,F_{n_y-1})\le d(y,F_{n_x-1})\le d(x,F_{n_x-1})+d(x,y)\le d(x,F_{n_x-1})+\tfrac13d(x,F_{n_x-1})\le 2d(x,F_{n_x-1}).$$
Then $$
\begin{aligned}
d(x,z)&\le d(x,y)+d(y,z)\le 3^{-k-1}\cdot(\min\{3^{-n_x},d(x,F_{n_x-1})\}+\min\{3^{-n_y},d(y,F_{n_y-1})\})\le\\
&\le3^{-k-1}\cdot(\min\{3^{-n_x},d(x,F_{n_x-1})\}+\min\{3^{-n_x},2d(x,F_{n_x-1})\})\le
3^{-k}\cdot\min\{3^{-n_x},d(x,F_{n_x-1})\}
\end{aligned}
$$
and hence $(x,z)\in V_k$.
So $V_{k+1}V_{k+1}\subset V_k$ for all $k\in\w$, which implies that the family $\{V_k\}_{k\in\w}$ is a base of a quasi-uniformity on $X$. Then $V_0\in\{V_k\}_{k\in\w}\subset\qU_X$ and $U\in\qU_X$.
\smallskip
Proposition~\ref{p3.6n} completes the proof.
\end{proof}

Following \cite{Vaugh} and \cite{BMZ} by $\mathfrak q_0$ we denote the smallest cardinality of a metrizable separable space which is not a $Q$-set. By Theorem 2 of \cite{BMZ}, $\mathfrak p\le \mathfrak q_0\le \min\{\mathfrak b,\log(\mathfrak c^+)\}$, which implies that $\mathfrak q_0=\mathfrak c$ under Martin's Axiom. We recall that $\mathfrak p$ is the smallest cardinality of a family $\A$ of infinite subsets of $\w$ such that for every finite subfamily $\F\subset\A$ the intersection $\bigcap\F$ is infinite and  for every infinite subset $A\subset\w$ there is a set $F\in\F$ such that $A\setminus F$ is infinite.

Denote by $\mathfrak q_\w$ the smallest cardinality of a metrizable separable space, which is not a $Q_\w$-set. Taking into account that each $Q$-set is a $Q_\w$-set and the real line is not a $Q_\omega$-set, we conclude that $\mathfrak p\le\mathfrak q_0\le \mathfrak q_\w\le\mathfrak c$, which implies that $\mathfrak p=\mathfrak q_\w=\mathfrak c$ under Martin's Axiom.

\begin{corollary} If $\mathfrak q_\w=\mathfrak c$, then for a metrizable separable space $X$ the following conditions are equivalent:
\begin{enumerate}
\item $\pU_X=\qU_X$;
\item $\pU_X=\pU_X^2$;
\item $|X|<\mathfrak c$.
\end{enumerate}
\end{corollary}

On the other hand, we have the following ZFC-result.

\begin{proposition}\label{p3.10n} Each metrizable space $X$ has $\qU_X=\pU_X^2$.
\end{proposition}

\begin{proof} Given any neighborhood assignment $U\in\pU_X$ we shall find an entourage $V\in\qU_X$ such that $V\subset UU$. Fix a metric $d\le 1$ generating the topology of $X$. For every $n\in\w$ consider the set $X_n=\{x\in X:B_d(x;3^{-n})\subset B(x;U)\}$ and observe that $X=\bigcup_{n\in\w}X_n=\bigcup_{n\in\w}\overline{X}_n$.  For every point $x\in X$ let $n_x=\min\{n\in\w:x\in \overline{X}_n\}$ Observe that $B_d(x;3^{-n_x})\subset B(x;U)$. Repeating the argument of the proof of the implication $(3)\Ra(1)$ in the proof of Proposition~\ref{p3.7n}, we can show that $\overline{X}_n\subset \{x\in X:B_d(x;3^{-n-1})\subset B(x;UU)\}$.

For every number $k\in\w$ consider the neighborhood assignment $V_k$ on $X$ assigning to each point $x\in X$ the open ball $B(x;V_k):=B_d(x;3^{-k-1}\cdot\min\{3^{-n_x},d(x;\overline{X}_{n_x-1})\})$. Here we assume that $X_{-1}=\overline{X}_{-1}=\emptyset$ and $d(x,\emptyset)=1$. It follows that $B(x;V_0)\subset B_d(x;3^{-n_x})\subset B(x;UU)$ and hence $V_0\subset UU$. Repeating the argument of the proof of Proposition~\ref{p3.7n}, we can show that $V_0\in\{V_k\}_{k\in\w}\subset\qU_X$ and hence $U\in\qU_X$.
\end{proof}

Propositions~\ref{p3.2n}--\ref{p3.7n} and \ref{p3.10n} imply that for metrizable spaces the diagram describing the inclusion relations between the canonical pre-uniformities $\U_X$, $\qU_X^{\pm n}$, $\qU^{\mp n}$,   $\qU_X^{\wedge n}$, $\qU^{\vee n}$, $\pU_X^{\pm n}$, $\pU^{\mp n}$,   $\pU_X^{\wedge n}$, $\pU^{\vee n}$ collapses to the following form.
{\small
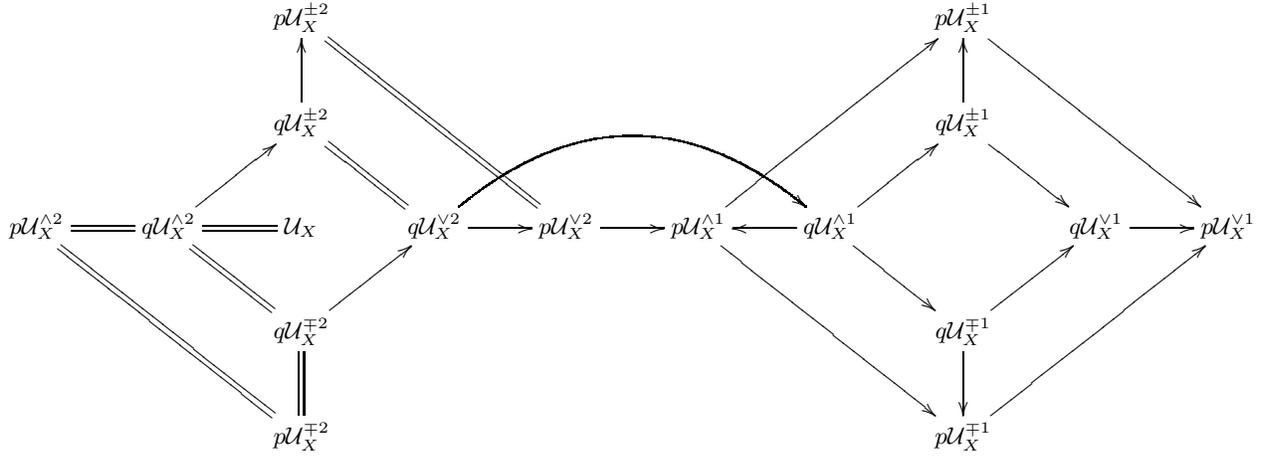
$$
\xymatrix{
&&\pU_X^{\pm 2}\ar@{=}[rrdd]&&&&&\pU_X^{\pm 1}\ar[rrdd]\\
&&\qU_X^{\pm 2}\ar@{=}[rd]\ar[u]&&&&&\qU_X^{\pm 1}\ar[rd]\ar[u]\\
\pU_X^{\wedge 2}\ar@{=}[rrdd]&\qU_X^{\wedge 2}\ar@{=}[l]\ar[ru]\ar@{=}[rd]&\U_X\ar@{=}[l]&\qU_X^{\vee 2}\ar[r]\ar@/^35pt/[rrr]&\pU_X^{\vee 2}\ar[r]&\pU_X^{\wedge 1}\ar[rruu]\ar[rrdd]&\qU_X^{\wedge 1}\ar[l]\ar[ru]\ar[rd]&&\qU_X^{\vee 1}\ar[r]&\pU_X^{\vee 1}\\
&&\qU_X^{\mp 2}\ar[ru]\ar@{=}[d]&&&&&\qU_X^{\mp 1}\ar[ru]\ar[d]\\
&&\pU_X^{\mp 2}&&&&&\pU_X^{\mp 1}\ar[rruu]
}
$$
}
\begin{example} For a $T_1$-space $X$ with a unique non-isolated point $\infty$ we get $\pU_X=\qU_X$ and $$\U_X=\pU^{\mp2}_X\ne \pU^{\pm2}=\pU^{\wedge 1}\ne \pU_X\ne\pU_X^{-1}\ne\pU^{\vee 1}=\U_{\dot X}$$
where $\dot X$ denotes $X$ endowed with the discrete topology.
\end{example}

\begin{proof} By Propositions~\ref{p3.2n} and \ref{p3.4n}, $\pU_X=\qU_X$ and $\U_X=\U^{\mp2}_X$. Denote by $\mathcal N_\infty$ the family of all neighborhood of the (unique) non-isolated point $\infty\in X$ and observe that
$$\begin{aligned}
&\pU_X=\big\langle \big\{\Delta_X\cup(\{\infty\}\times O_\infty):O_\infty\in\mathcal N_\infty\big\}\big\rangle,\\
&\pU_X^{-1}=\big\langle \big\{\Delta_X\cup(O_\infty\times\{\infty\}):O_\infty\in\mathcal N_\infty\big\}\big\rangle,\\
&\pU_X^{\vee1}=\big\langle \big\{\Delta_X\big\}\big\rangle=\U_{\dot X},\\
&\pU_X^{\wedge 1}=\big\langle \big\{\Delta_X\cup(\{\infty\}\times O_\infty)\cup(O_\infty\times\{\infty\}):O_\infty\in\mathcal N_\infty\big\}\big\rangle=\U_X^{\pm2},\\
&\pU^{\mp2}=\big\langle \big\{\Delta_X\cup(O_\infty\times O_\infty):O_\infty\in\mathcal N_\infty\big\}\big\rangle=\U_X,
\end{aligned}
$$
which yields the required (in)equalities.
\end{proof}

\section{Verbal covering properties of topological spaces}\label{s3}

Cardinal characteristics of the pre-uniformities $\pU_X$, $\qU_X$ and $\U_X$ or (Boolean operations over) their verbal powers can be considered as cardinal characteristics of the topological space $X$.

Namely, for any binary word $v\in \{+,-\}^{<\w}$ and any topological space $X$ consider the cardinals
$$
\begin{aligned}
&\ell^v(X):=\ell(\pU_X^v),& &\bar\ell^v(X):=\bar\ell(\pU_X^v),& &L^v(X):=L(\pU_X^v),& &\bar L^v(X):=\bar L(\pU_X^v),&\\
&q\ell^v(X):=\ell(\qU_X^v),&  &q\bar\ell^v(X):=\bar\ell(\qU_X^v),& &qL^v(X):=L(\qU_X^v),&  &q\bar L^v(X):=\bar L(\qU_X^v),&\\
&u\ell^v(X):=\ell(\U_X^v),&  &u\bar\ell^v(X):=\bar\ell(\U_X^v),& &uL^v(X):=L(\U_X^v),&  &u\bar L^v(X):=\bar L(\U_X^v).&
\end{aligned}
$$
We also put $u\ell(X):=\ell(\U_X)$ and $uL(X)=L(\U_X)$.

Taking into account that $\U_X^2=\U_X=\U_X^{-1}$  we conclude that  $$u\ell^v(X)=u\bar\ell^v(X)=u\ell(X)\mbox{ and }uL^v(X)=u\bar L^v(X)=u L(X)$$
for every binary word $v\in\{+,-\}^{<\w}\setminus\{\emptyset\}$.
So, all cardinal characteristics $u\ell^v(X)$, $u\bar \ell^v(X)$, $uL^v(X)$, $u\bar L^v(X)$ with $v\ne\emptyset$ collapse to two cardinal characteristics $u\ell(X)$ and $uL(X)$ (of which $uL$ determines $u\ell$).

On the other hand, the equality $\qU_X^2=\qU_X$ implies that for every word $v\in\{+,-\}^{<\w}$ the cardinal $q\ell^v(X)$ (resp. $q\bar \ell^v(X)$) is equal to $q\ell^{\pm n}(X)$ or $q\ell^{\mp n}(X)$ (resp. $q\bar \ell^{\pm n}(X)$ or $q\bar \ell^{\mp n}(X)$) for some $n\in\w$.

For every $n\in\IN$ and a topological space $X$ consider the cardinals
$$
\begin{aligned}
&\ell^{\wedge n}(X):=\ell^{\wedge n}(p\U_X),&\;\;&q\ell^{\wedge n}(X):=\ell^{\wedge n}(q\U_X)&\\
&\ell^{\vee n}(X):=\ell^{\vee n}(p\U_X^{\pm n})&,\;\;&q\ell^{\vee n}(X):=\ell^{\vee n}(q\U_X)&
\end{aligned}
$$ and $$\ell^{\w}(X)=\min\{\ell^v(X):v\in\{+,-\}^{<\w}\},\;\;
q\ell^{\w}(X)=\min\{q\ell^v(X):v\in\{+,-\}^{<\w}\}.$$

The diagram drawn at the beginning of Section~\ref{s3n} and the monotonicity of the boundedness number $\ell$ yield the following diagram describing the inequalities between the cardinal characteristics $u\ell$, $\ell^{\pm n}$, $\ell^{\mp n}$, $\ell^{\wedge n}$, $\ell^{\vee n}$, $q\ell^{\pm n}$, $q\ell^{\mp n}$, $q\ell^{\wedge n}$, $q\ell^{\vee n}$ for $n\in\IN$.
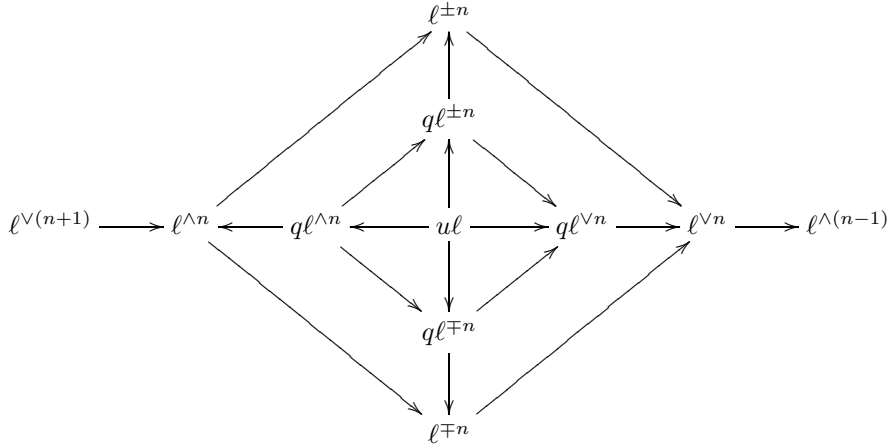
$$
\xymatrix{
&&&\ell^{\pm n}\ar[rrdd]\\
&&&q\ell^{\pm n}\ar[rd]\ar[u]\\
\ell^{\vee(n+1)}\ar[r]&\ell^{\wedge n}\ar[rruu]\ar[rrdd]&q\ell^{\wedge n}\ar[l]\ar[ru]\ar[rd]&u\ell\ar[r]\ar[l]\ar[u]\ar[d]&q\ell^{\vee n}\ar[r]&\ell^{\vee n}\ar[r]&\ell^{\wedge (n-1)}\\
&&&q\ell^{\mp n}\ar[ru]\ar[d]\\
&&&\ell^{\mp n}\ar[rruu]
}
$$

\begin{definition}
Let $v\in\{+,-\}^{<\w}$ be a binary word. A topological space $X$ is defined to be
\begin{itemize}
\item {\em $v$-compact} if $L^v(X)\le\w$;
\item {\em $v$-Lindel\"of} if $\ell^v(X)\le\w$;
\item {\em weakly $v$-Lindel\"of} if $\bar\ell^v(X)\le\w$.
\end{itemize}
\end{definition}

Observe that a topological space $X$ is compact (resp. Lindel\"of, weakly Lindel\"of) if and only if $X$ is $+$-compact, (resp. $+$-Lindel\"of, weakly $+$-Lindel\"of).

\begin{problem} Given two distinct binary words $v,w\in\{+,-\}^{<\w}$ study the relations between the cardinal characteristics $\ell^v$, $\bar\ell^v$, $\ell^w$, $\bar \ell^w$. Are these cardinal characteristics pairwise distinct?
\end{problem}

Observe that for any neighborhood assignment $U$ on a topological space $X$ and any subset $A\subset X$ we get $\bar A\subset B(A;U^{-1})$. This implies that for any binary word $v\in\{+,-\}^{<\w}$ we get the following diagram:
$$
\xymatrix{
u\ell\ar[r]&q\bar\ell^{v\mbox{-}}\ar@{=}[r]\ar[d]&q\ell^{v\mbox{-}}\ar[r]\ar[d]&q\bar \ell^v\ar[r]\ar[d]&q\ell^v\ar[d]\\
&\bar\ell^{v\mbox{-}}\ar[r]&\ell^{v\mbox{-}}\ar[r]&\bar\ell^v\ar[r]&\ell^v.
}
$$

Using Lemma~\ref{l2.2} it is easy to prove the following proposition showing that all star-covering properties of topological spaces can be expressed by the cardinal characteristics $\ell^v$, $\bar\ell^v$, for suitable alternating words $v$.

\begin{proposition}\label{p1.5} For every $n\in\w$ we have the equalities:
$$l^{*n}=\ell^{\mp 2n},\;\;\bar l^{*n}=\bar\ell^{\mp 2n},\;\;l^{*n\frac12}=\ell^{\pm(2n+1)},\;\;\bar l^{*n\frac12}=\bar\ell^{\pm(2n+1)},\mbox{ \ and}$$
$$L^{*n}=L^{\mp 2n},\;\;\bar L^{*n}=\bar L^{\mp 2n},\;\;L^{*n\frac12}=L^{\pm(2n+1)},\;\;\bar L^{*n\frac12}=\bar L^{\pm(2n+1)}.$$
\end{proposition}

The initial cardinal characteristics $\ell^{\pm1}$, $\ell^{\mp1}$, $\ell^{\vee1}$ and $q\ell^{\pm1}$, $q\ell^{\mp1}$, $q\ell^{\vee1}$ have nice inheritance properties.

\begin{proposition}\label{p1.6} Any closed subspace $F$ of a topological space $X$ has
\begin{enumerate}
\item $\ell^{\pm1}(F)\le \ell^{\pm1}(X)$ and $\ell^{\vee1}(F)\le \ell^{\vee1}(X)$;
\item $q\ell^{\pm1}(F)\le q\ell^{\pm1}(X)$ and $q\ell^{\vee1}(F)\le q\ell^{\vee1}(X)$.
\end{enumerate}
\end{proposition}

\begin{proof} Given a neighborhood assignment $V\in\pU_F$ on $F$, consider the neighborhood assignment $\tilde V=V\cup (X\times (X\setminus F))$ on $X$ and observe that $B(x;\tilde V)\cap F=B(x;V)$ for every $x\in F$. By definitions of the cardinals $\ell^{\pm1}(X)$ and $\ell^{\vee1}(X)$, there are subsets $A,A_1\subset X$ of cardinality $|A|\le \ell^{\pm1}(X)$ and $|A_1|\le \ell^{\vee1}(X)$ such that $X=B(A;\tilde V)$ and $X=B(A_1;\tilde V\cap \tilde V^{-1})$. Then $A\cap F$ is a subset of cardinality $|A\cap F|\le |A|\le \ell^{\pm1}(X)$ such that $F=B(A\cap F;V)=F\cap B(A\cap F;\tilde V)$, witnessing that $\ell^{\pm1}(F)\le \ell^{\pm1}(X)$. On the other hand, $A_1\cap F$ is a subset of cardinality $|A_1\cap F|\le |A_1|\le \ell^{\vee1}(X)$ such that $F\subset B(A_1\cap F;V\cap V^{-1})$.
Indeed, for every $x\in F$ there is a point $a\in A_1$ such that $x\in B(a;\tilde V\cap \tilde V^{-1})$. Then $B(a;\tilde V)\cap F\ne\emptyset$, which implies $a\in F$. The choice of the neighborhood assignment $\tilde V$ guarantees that $x\in F\cap B(a;\tilde V)=B(a;V)$ and $a\in F\cap B(x;\tilde V)=B(x;V)$. So, $x\in B(a;V\cap V^{-1})$ and the set $A_1\cap F$ witnesses that $\ell^{\vee1}(F)\le \ell^{\vee1}(X)$.

If $V\in \qU_F$, then $\tilde V\in q\U_X$. Indeed, for the neighborhood assignment $V$ we could find a sequence of neighborhood assignments $(V_n)_{n\in\w}$ in $\qU_F$ such that $V_0=V$ and $V_{n+1}^2\subset V_n$ for every $n\in\w$. For every $n\in\w$ consider the neighborhood assignment $\tilde V_n=V_n\cup(X\times (X\setminus F))$ on $X$. We claim that $\tilde V_{n}^2\subset \tilde V_{n-1}$ for all $n>0$. Given any number $n>0$ and pair $(x,y)\in\tilde V_{n}^2$, we should check that $(x,y)\in\tilde V_{n-1}$. This is trivially true if $y\notin F$. So, we assume that $y\in F$. Since $(x,y)\in\tilde V_n^2$, there is a point $z\in X$ such that $(x,z),(z,y)\in\tilde V_{n}$. Since $y\in F$, the definition of $\tilde V_n$ implies that $z\in F$ and $x\in F$. Then $(x,z),(z,y)\in (F\times F)\times\tilde V_n=V_n$ and hence $(x,y)\in V_n^2\subset V_{n-1}^2\subset\tilde V_{n-1}$. It follows that the set $\{V_n\}_{n\in\w}$ generates a quasi-uniformity $\V=\{V\subset X\times X:\tilde V_n\subset V$ for some $n\in\w\}$ consisting of neighborhood assignments on $X$.

Consequently, $\tilde V=\tilde V_0\in \V\subset\qU_X$. In this case we can additionally assume that the subsets $A,A_1\subset X$ have cardinality $|A|\le q\ell^{\pm1}(X)$ and $|A_1|\le q\ell^{\vee1}(X)$, which implies $q\ell^{\pm1}(F)\le q\ell^{\pm1}(X)$ and $q\ell^{\vee1}(F)\le q\ell^{\vee1}(X)$.
\end{proof}

\begin{proposition}\label{p1.7} Any open subspace $U$ of a topological space $X$ has
\begin{enumerate}
\item $\ell^{\mp1}(U)\le \ell^{\mp1}(X)$ and $\ell^{\vee1}(U)\le \ell^{\vee1}(X)$;
\item $q\ell^{\mp1}(U)\le q\ell^{\pm1}(X)$ and $q\ell^{\vee1}(U)\le q\ell^{\vee1}(X)$.
\end{enumerate}
\end{proposition}

\begin{proof} Given any neighborhood assignment $V$ on $U$, observe that $\tilde V=V\cup \big((X\setminus U)\times X\big)$ is a neighborhood assignment on $X$ such that $B(x;\tilde V)=B(x;V)$ for every $x\in U$. By the definitions of $\ell^{\mp1}(X)$ and $\ell^{\vee1}(X)$, there are  subsets $A\subset X$ and $A_1\subset X$ of cardinality $|A|\le \ell^{\mp1}(X)$ and $|A_1|\le \ell^{\vee1}(X)$ such that $B(x;\tilde V)\cap A\ne\emptyset$ and $B(x;\tilde V\cap\tilde V^{-1})\cap A_1\ne\emptyset$ for every $x\in U$. Then $A\cap U$ and $A_1\cap U$ are subsets of cardinality $|A\cap U|\le|A|\le \ell^{\mp1}(X)$ and $|A_1\cap U|\le|A_1|\le \ell^{\vee 1}(X)$ such that $B(x;V)\cap (A\cap U)\ne\emptyset$ and $B(x;V\cap V^{-1})\cap (A_1\cap U)\ne \emptyset$ for all $x\in U$. This witnesses that $\ell^{\mp1}(U)\le \ell^{\mp1}(X)$ and $\ell^{\vee 1}(U)\le \ell^{\vee1}(X)$.

If $V\in\qU_U$, then by analogy with the proof of Proposition~\ref{p1.6}, we can show that the neighborhood assignment $\tilde V$ belongs to the universal quasi-uniformity $\qU_X$. In this case we can assume that the sets $A,A_1$ have cardinality $|A|\le q\ell^{\mp1}(X)$ and $|A_1|\le q\ell^{\vee1}(X)$, which implies $q\ell^{\mp1}(U)\le q\ell^{\mp1}(X)$ and $q\ell^{\vee1}(U)\le q\ell^{\vee1}(X)$.
\end{proof}

\begin{proposition}\label{p1.8} Let $X$ be a topological space. Then
\begin{enumerate}
\item $\ell^{\wedge1}(X)\le s(X)\le q\ell^{\vee1}(X)\le \ell^{\vee 1}(X)\le nw(X)$;
\item $de(X)\le q\ell^{\pm1}(X)\le \ell^{\pm1}(X)=l(X)$;
\item $c(X)\le q\ell^{\mp1}(X)\le \ell^{\mp1}(X)\le d(X)$;
\item If $X$ is quasi-regular, then $\bar\ell^{\pm 3}(X)=\bar l^{*1\kern-1pt\frac12}(X)=\ell^{\w}(X)=\dc(X)$;
\item If $X$ is completely regular, then $q\bar\ell^{\pm 3}(X)=q\ell^{\w}(X)=u\ell(X)=\dc(X)$.
\end{enumerate}
\end{proposition}

\begin{proof} 1. First we prove that $\ell^{\wedge 1}(X)\le s(X)$. Given any neighborhood assignment $V$ on $X$ we need to find a subset $A\subset X$ of cardinality $|A|\le s(X)$ such that $X=B(A;V\cup V^{-1})$. Using Zorn's Lemma, choose a maximal subset $A\subset X$ such that $y\notin B(x;V\cup V^{-1})$ for any distinct points $x,y\in X$. Taking into account that for every $x\in A$ the set $B(x;V\cup V^{-1})$ is a neighborhood of $x$, we conclude that the space $A$ is discrete and hence has cardinality $|A|\le s(X)$.

To see that $s(X)\le q\ell^{\vee1}(X)$, take any discrete subspace $D\subset X$. For every point $x\in D$ choose an open set $O_x\subset X$ such that $O_x\cap D=\{x\}$ and put $U_x=O_x\cap\overline{D}$.  It follows that $U_x\subset\bar x$ where $\overline{x}$ is the closure of the singleton $\{x\}$ in $X$.
Observe that for any distinct points $x,y\in D$ we get $U_x\cap U_y\subset O_x\cap O_y\cap\bar x\cap \bar y=\emptyset$. So, $(U_x)_{x\in D}$ is a disjoint family of open sets in $\overline{D}$ and $U=\bigcup_{x\in D}U_x$ is an open dense set in $\overline{D}$. By Propositions~\ref{p1.6} and \ref{p1.7} we conclude that $|D|\le q\ell^{\vee1}(U)\le q\ell^{\vee1}(\overline{D})\le q\ell^{\vee1}(X)$, which implies $s(X)\le q\ell^{\vee 1}(X)$.

The inequality $q\ell^{\vee1}(X)\le \ell^{\vee1}(X)$ is trivial. To see that $\ell^{\vee1}(X)\le nw(X)$,  fix a network $\mathcal N$ of the topology of $X$ of cardinality $|\mathcal N|=nw(X)$. Given a neighborhood assignment $V$ on $X$, for every $x\in X$ find a set $N_x\in\mathcal N$ such that $x\in N_x\subset B(x;V)$. Since $|\{N_x:x\in X\}|\le|\mathcal N|=nw(X)$, we can choose a subset $A\subset X$ of cardinality $|A|\le nw(X)$ such that $\{N_x:x\in A\}=\{N_x:x\in X\}$. We claim that $X=B(A;V\cap V^{-1})$. Indeed, for every point $x\in X$ we can find a point $a\in A$ such that $N_a=N_x$. Then $x\in N_x=N_a\subset B(a;V)$ and $a\in N_a=N_x\subset B(x;V)$, which implies $x\in B(a;V\cap V^{-1})$. So, $\ell^{\vee1}(X)\le nw(X)$.
\smallskip

2. The equality $\ell^{\pm1}(X)=l(X)$ follows from Proposition~\ref{p1.5} and $q\ell^{\pm1}(X)\le \ell^{\pm1}(X)$ is trivial. To see that $de(X)\le q\ell^{\pm1}(X)$, take any discrete family $\mathcal D$ of subsets in $X$. Replacing each set $D\in\mathcal D$ by its closure $\overline{D}$, we can assume that $\mathcal D$ consists of closed sets. Then its union $\bigcup\mathcal D$ is a closed set in $X$ and  Proposition~\ref{p1.6} guarantees that $|\mathcal D|\le q\ell^{\pm1}(\cup\mathcal D)\le q\ell^{\pm1}(X)$. So, $de(X)\le q\ell^{\pm1}(X)$.
\smallskip

3. The inequality $\ell^{\mp1}(X)\le d(X)$ trivially follows from the definitions of the cardinal invariants $\ell^{\mp1}$ and $d$. To see that $c(X)\le q\ell^{\mp1}(X)$, take any disjoint family $\U$ of non-empty open sets in $X$ and applying Proposition~\ref{p1.7}, conclude that $|\U|\le q\ell^{\mp1}(\bigcup\U)\le q\ell^{\mp1}(X)$, which implies that $c(X)\le q\ell^{\mp1}(X)$.
\smallskip

4. If $X$ is quasi-regular, then the equality $\dc(X)=l^{*\w}(X)=\bar l^{*1\kern-1pt\frac12}(X)=\bar\ell^{\pm 3}(X)=\ell^{\w}(X)$ follows from Propositions~\ref{p1.2} and \ref{p1.5}.
\smallskip

5. Next, assume that $X$ is completely regular. Then we get the inequality
$u\ell(X)\le q\ell^{\w}(X)\le q\bar\ell^{\pm3}(X)\le\bar\ell^{\pm3}(X)=\dc(X)$. These inequalities will turn into equalities if we prove that $\dc(X)\le u\ell(X)$. Assuming that $u\ell(X)<\dc(X)$, for the cardinal $\kappa=u\ell(X)$ we can find a discrete family $(U_\alpha)_{\alpha\in\kappa^+}$ of non-empty open sets on $X$. In each set $U_\alpha$ fix a point $x_\alpha$. Since $X$ is completely regular, for every $\alpha\in\kappa^+$   we can choose a continuous function $f_\alpha:X\to [0,1]$ such that $f_\alpha(x_\alpha)=1$ and $f_\alpha^{-1}\big([0,1)\big)\subset U_\alpha$. The functions $(f_\alpha)_{\alpha\in\kappa^+}$ determine a continuous pseudometric $d:X\times X\to[0,1]$ defined by the formula $d(x,y)=\sum_{\alpha\in \kappa^+}|f_\alpha(x)-f_\alpha(y)|$. The pseudometric $d$ determines a uniformity generated by the base consisting of entourages $[d]_{<\e}=\{(x,y)\in X\times X:d(x,y)<\e\}$, which belong to the universal uniformity $\U_X$ on $X$. Observe that for any subset $A\subset X$ of cardinality $|A|\le\kappa$
there is an index $\alpha\in\kappa^+$ such that $U_\alpha\cap A=\emptyset$. Then for the entourage $U=[d]_{<1}\in\U_X$ we get $x_\alpha\notin B(A;U)$, which implies that $u\ell(X)=\ell(\U_X)>\kappa=u\ell(X)$ and this is a desired contradiction.
\end{proof}

Taking into account Propositions~\ref{p1.2}, \ref{p1.5} and \ref{p1.8}, we see that for a quasi-regular space $X$ the cardinal characteristics $\ell^{\pm n}$, $\ell^{\mp n}$, $\bar \ell^{\pm n}$, $\bar\ell^{\mp n}$, $\ell^{\wedge n}$ and $\ell^{\vee n}$ relate with other cardinal characteristics of topological spaces as follows.

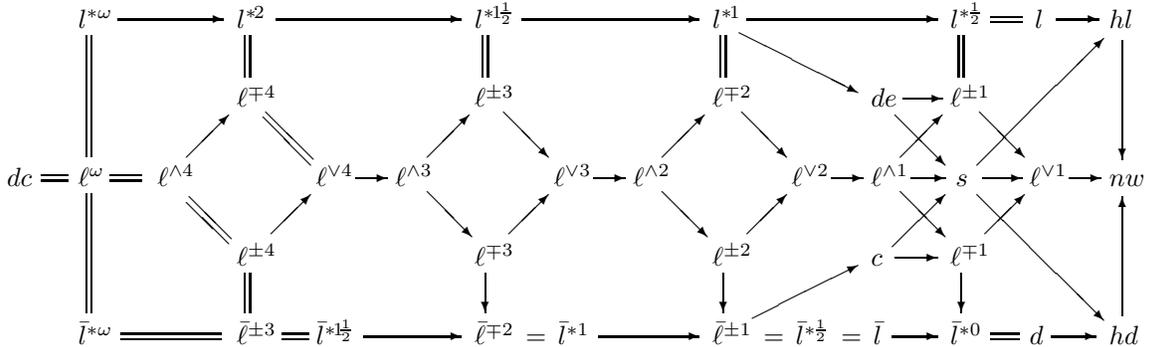
\begin{picture}(400,150)(-35,-10)
\put(0,0){$\bar l^{*\w}$}
\put(3,12){\line(0,1){45}}
\put(5,12){\line(0,1){45}}
\put(16,2){\line(1,0){38}}
\put(16,4){\line(1,0){38}}
\put(60,0){$\bar\ell^{\pm3}$}
\put(77,2){\line(1,0){10}}
\put(77,4){\line(1,0){10}}
\put(63,12){\line(0,1){15}}
\put(65,12){\line(0,1){15}}
\put(90,0){$\bar l^{*\kern-1pt 1\kern-2pt \frac12}$}
\put(108,3){\vector(1,0){37}}
\put(150,0){$\bar\ell^{\mp 2}\,=\,\bar l^{*1}$}
\put(197,3){\vector(1,0){38}}
\put(240,0){$\bar\ell^{\pm 1}\,=\,\bar l^{*\kern-1pt \frac12}\,=\,\bar l$}
\put(255,10){\vector(2,1){40}}
\put(308,3){\vector(1,0){17}}
\put(330,0){$\bar l^{*0}$}
\put(345,2){\line(1,0){11}}
\put(345,4){\line(1,0){11}}
\put(360,0){$d$}
\put(368,3){\vector(1,0){17}}
\put(390,0){$hd$}
\put(395,11){\vector(0,1){45}}

\put(0,120){$l^{*\w}$}
\put(15,123){\vector(1,0){40}}
\put(60,120){$l^{*\kern-1pt 2}$}
\put(75,123){\vector(1,0){70}}
\put(150,120){$l^{*\kern-1pt 1\kern-1.5pt \frac12}$}
\put(168,123){\vector(1,0){67}}
\put(240,120){$l^{*\kern-1pt 1}$}
\put(250,118){\vector(2,-1){45}}
\put(253,123){\vector(1,0){72}}
\put(330,120){$l^{*\kern-1pt \frac12}$}
\put(345,122){\line(1,0){12}}
\put(345,124){\line(1,0){12}}
\put(362,120){$l$}
\put(370,123){\vector(1,0){16}}
\put(390,120){$hl$}
\put(395,115){\vector(0,-1){45}}

\put(60,30){$\ell^{\pm4}$}
\put(73,42){\vector(1,1){15}}
\put(150,30){$\ell^{\mp 3}$}
\put(163,42){\vector(1,1){15}}
\put(154,27){\vector(0,-1){15}}
\put(240,30){$\ell^{\pm2}$}
\put(253,42){\vector(1,1){15}}
\put(244,27){\vector(0,-1){15}}
\put(300,30){$c$}
\put(308,37){\vector(1,1){20}}
\put(309,33){\vector(1,0){16}}
\put(330,30){$\ell^{\mp1}$}
\put(343,42){\vector(1,1){15}}
\put(334,27){\vector(0,-1){15}}

\put(60,90){$\ell^{\mp4}$}
\put(71,88){\line(1,-1){18}}
\put(70,86){\line(1,-1){18}}
\put(63,101){\line(0,1){16}}
\put(65,101){\line(0,1){16}}
\put(150,90){$\ell^{\pm 3}$}
\put(161,88){\vector(1,-1){17}}
\put(153,101){\line(0,1){16}}
\put(155,101){\line(0,1){16}}
\put(240,90){$\ell^{\mp2}$}
\put(251,88){\vector(1,-1){17}}
\put(243,101){\line(0,1){16}}
\put(245,101){\line(0,1){16}}
\put(300,90){$de$}
\put(309,87){\vector(1,-1){19}}
\put(312,93){\vector(1,0){15}}
\put(330,90){$\ell^{\pm1}$}
\put(341,88){\vector(1,-1){17}}
\put(333,101){\line(0,1){16}}
\put(335,101){\line(0,1){16}}

\put(-27,60){$dc$}
\put(-14,62){\line(1,0){10}}
\put(-14,64){\line(1,0){10}}
\put(0,60){$\ell^{\w}$}
\put(3,71){\line(0,1){46}}
\put(5,71){\line(0,1){46}}
\put(12,62){\line(1,0){12}}
\put(12,64){\line(1,0){12}}
\put(30,60){$\ell^{{\wedge}4}$}
\put(42,56){\line(1,-1){16}}
\put(41,54){\line(1,-1){16}}
\put(41,71){\vector(1,1){16}}
\put(90,60){$\ell^{\vee4}$}
\put(105,63){\vector(1,0){12}}
\put(120,60){$\ell^{{\wedge}3}$}
\put(132,71){\vector(1,1){16}}
\put(132,58){\vector(1,-1){17}}
\put(180,60){$\ell^{\vee3}$}
\put(195,63){\vector(1,0){12}}
\put(210,60){$\ell^{{\wedge}2}$}
\put(221,71){\vector(1,1){16}}
\put(221,58){\vector(1,-1){17}}
\put(270,60){$\ell^{\vee2}$}
\put(285,63){\vector(1,0){12}}
\put(300,60){$\ell^{{\wedge}1}$}
\put(311,71){\vector(1,1){16}}
\put(311,58){\vector(1,-1){17}}
\put(315,63){\vector(1,0){13}}
\put(332,60){$s$}
\put(340,68){\vector(1,1){48}}
\put(340,57){\vector(1,-1){48}}
\put(342,63){\vector(1,0){15}}
\put(360,60){$\ell^{\vee1}$}
\put(375,63){\vector(1,0){12}}
\put(390,60){$nw$}
\end{picture}

For Tychonoff spaces we can add to this diagram the cardinal characteristics $q\ell^{\pm n}$, $q\ell^{\mp n}$, $q\ell^{\vee n}$, and $u\ell$ and obtain a  more complex diagram.

{\small
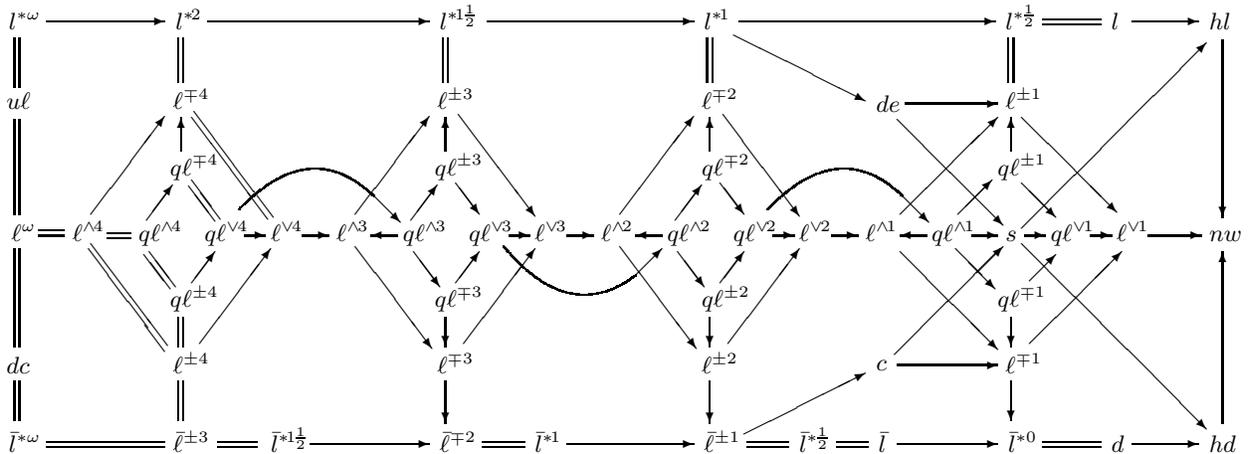
\begin{picture}(400,180)(5,-2)
\put(1,0){$\bar l^{*\w}$}
\put(3,12){\line(0,1){15}}
\put(5,12){\line(0,1){15}}
\put(15,2){\line(1,0){45}}
\put(15,4){\line(1,0){45}}
\put(63,0){$\bar\ell^{\pm3}$}
\put(80,2){\line(1,0){15}}
\put(80,4){\line(1,0){15}}
\put(65,12){\line(0,1){15}}
\put(67,12){\line(0,1){15}}
\put(100,0){$\bar l^{*\kern-1pt 1\kern-1pt \frac12}$}
\put(115,3){\vector(1,0){45}}
\put(164,0){$\bar\ell^{\mp 2}$}
\put(180,2){\line(1,0){16}}
\put(180,4){\line(1,0){16}}
\put(200,0){$\bar l^{*\kern-1pt 1}$}
\put(212,3){\vector(1,0){48}}
\put(264,0){$\bar\ell^{\pm 1}$}
\put(280,2){\line(1,0){16}}
\put(280,4){\line(1,0){16}}
\put(300,0){$\bar l^{*\kern-1pt \frac12}$}
\put(314,2){\line(1,0){12}}
\put(314,4){\line(1,0){12}}
\put(330,0){$\bar l$}
\put(337,3){\vector(1,0){38}}
\put(279,7){\vector(2,1){46}}

\put(379,0){$\bar l^{*0}$}
\put(392,2){\line(1,0){22}}
\put(392,4){\line(1,0){22}}
\put(418,0){$d$}
\put(426,3){\vector(1,0){25}}
\put(455,0){$hd$}
\put(460,11){\vector(0,1){65}}

\put(0,30){$dc$}
\put(3,40){\line(0,1){37}}
\put(5,40){\line(0,1){37}}
\put(63,30){$\ell^{\pm4}$}
\put(73,41){\vector(3,4){27}}
\put(163,30){$\ell^{\mp 3}$}
\put(173,41){\vector(3,4){27}}
\put(166,27){\vector(0,-1){16}}
\put(263,30){$\ell^{\pm2}$}
\put(273,41){\vector(3,4){27}}

\put(266,27){\vector(0,-1){16}}
\put(329,31){$c$}
\put(336,37){\vector(1,1){41}}
\put(337,33){\vector(1,0){38}}
\put(378,30){$\ell^{\mp1}$}
\put(387,41){\vector(1,1){35}}
\put(380,27){\vector(0,-1){15}}

\put(62,55){$q\ell^{\pm4}$}
\put(70,64){\vector(3,4){9}}
\put(65,51){\line(0,-1){11}}
\put(67,51){\line(0,-1){11}}
\put(162,55){$q\ell^{\mp 3}$}
\put(170,64){\vector(3,4){9}}
\put(166,51){\vector(0,-1){11}}
\put(263,55){$q\ell^{\pm2}$}
\put(270,64){\vector(3,4){9}}
\put(266,51){\vector(0,-1){11}}
\put(375,55){$q\ell^{\mp1}$}
\put(387,65){\vector(1,1){10}}
\put(380,51){\vector(0,-1){11}}

\put(2,80){$\ell^{\w}$}
\put(3,90){\line(0,1){36}}
\put(5,90){\line(0,1){36}}
\put(12,82){\line(1,0){10}}
\put(12,84){\line(1,0){10}}
\put(25,80){$\ell^{\wedge\kern-1pt 4}$}
\put(32,90){\vector(3,4){28}}
\put(32,79){\line(3,-4){29}}
\put(31,77){\line(3,-4){29}}
\put(50,80){$q\ell^{\wedge\kern-1pt 4}$}
\put(53,88){\vector(3,4){10}}
\put(54,76){\line(3,-4){10}}
\put(52,75){\line(3,-4){10}}
\put(47,83){\line(-1,0){9}}
\put(47,81){\line(-1,0){9}}
\put(75,80){$q\ell^{\vee\kern-1pt 4}$}
\qbezier(88,92)(115,120)(139,97)
\put(139,97){\vector(1,-1){10}}
\put(90,82){\vector(1,0){8}}
\put(100,80){$\ell^{\vee\kern-1pt 4}$}
\put(112,82){\vector(1,0){10}}
\put(125,80){$\ell^{\wedge\kern-1pt 3}$}
\put(132,90){\vector(3,4){28}}
\put(132,78){\vector(3,-4){29}}
\put(150,80){$q\ell^{\wedge\kern-1pt 3}$}
\put(153,88){\vector(3,4){10}}
\put(154,76){\vector(3,-4){10}}
\put(148,82){\vector(-1,0){10}}
\put(175,80){$q\ell^{\vee\kern-1pt 3}$}
\put(190,82){\vector(1,0){8}}
\put(200,80){$\ell^{\vee\kern-1pt 3}$}
\put(212,82){\vector(1,0){10}}
\qbezier(188,78)(213,48)(238,67)
\put(238,67){\vector(1,1){10}}

\put(225,80){$\ell^{\wedge\kern-1pt 2}$}
\put(232,90){\vector(3,4){28}}
\put(232,78){\vector(3,-4){29}}
\put(250,80){$q\ell^{\wedge\kern-1pt 2}$}
\put(253,88){\vector(3,4){10}}
\put(254,76){\vector(3,-4){10}}
\put(248,82){\vector(-1,0){10}}
\put(275,80){$q\ell^{\vee\kern-1pt 2}$}
\qbezier(288,92)(315,120)(339,97)
\put(339,97){\vector(1,-1){10}}
\put(290,82){\vector(1,0){8}}
\put(300,80){$\ell^{\vee\kern-1pt 2}$}
\put(312,82){\vector(1,0){10}}
\put(325,80){$\ell^{\wedge\kern-1pt 1}$}
\put(338,90){\vector(1,1){37}}
\put(338,76){\vector(1,-1){37}}
\put(350,80){$q\ell^{\wedge\kern-1pt 1}$}
\put(360,90){\vector(1,1){13}}
\put(359,77){\vector(1,-1){15}}
\put(348,82){\vector(-1,0){10}}
\put(365,82){\vector(1,0){10}}
\put(378,80){$s$}
\put(384,85){\vector(1,1){70}}
\put(384,79){\vector(1,-1){70}}
\put(385,82){\vector(1,0){8}}
\put(395,80){$q\ell^{\vee\kern-1pt 1}$}
\put(410,82){\vector(1,0){8}}
\put(420,80){$\ell^{\vee\kern-1pt 1}$}
\put(432,82){\vector(1,0){20}}
\put(455,80){$nw$}

\put(62,105){$q\ell^{\mp4}$}
\put(70,103){\line(3,-4){11}}
\put(69,101){\line(3,-4){10}}
\put(66,114){\vector(0,1){13}}
\put(162,105){$q\ell^{\pm 3}$}
\put(170,102){\vector(3,-4){9}}
\put(166,114){\vector(0,1){13}}
\put(263,105){$q\ell^{\mp2}$}
\put(270,102){\vector(3,-4){9}}
\put(266,114){\vector(0,1){13}}
\put(375,105){$q\ell^{\pm1}$}
\put(384,102){\vector(1,-1){13}}
\put(380,114){\vector(0,1){13}}

\put(0,130){$u\ell$}
\put(3,138){\line(0,1){19}}
\put(5,138){\line(0,1){19}}
\put(3,40){\line(0,1){37}}
\put(5,40){\line(0,1){37}}
\put(63,130){$\ell^{\mp4}$}
\put(65,140){\line(0,1){17}}
\put(67,140){\line(0,1){17}}
\put(71,129){\line(3,-4){30}}
\put(70,127){\line(3,-4){29}}
\put(163,130){$\ell^{\pm 3}$}
\put(165,140){\line(0,1){17}}
\put(167,140){\line(0,1){17}}

\put(171,128){\vector(3,-4){29}}
\put(263,130){$\ell^{\mp2}$}
\put(265,140){\line(0,1){17}}
\put(267,140){\line(0,1){17}}

\put(271,128){\vector(3,-4){29}}

\put(329,129){$de$}
\put(337,126){\vector(1,-1){40}}
\put(340,132){\vector(1,0){35}}
\put(378,130){$\ell^{\pm1}$}
\put(379,140){\line(0,1){17}}
\put(381,140){\line(0,1){17}}

\put(384,127){\vector(1,-1){36}}

\put(1,160){$l^{*\w}$}
\put(15,163){\vector(1,0){45}}
\put(64,160){$l^{*\kern-1pt 2}$}
\put(76,163){\vector(1,0){84}}
\put(164,160){$l^{*\kern-1pt 1\kern-1pt \frac12}$}
\put(182,163){\vector(1,0){78}}
\put(264,160){$l^{*\kern-1pt 1}$}
\put(275,158){\vector(2,-1){50}}
\put(277,163){\vector(1,0){97}}
\put(378,160){$l^{*\kern-1pt \frac12}$}
\put(392,162){\line(1,0){22}}
\put(392,164){\line(1,0){22}}
\put(418,160){$l$}
\put(426,163){\vector(1,0){25}}
\put(455,160){$hl$}
\put(460,156){\vector(0,-1){67}}
\end{picture}
}

\begin{problem}\label{pr1} Study the cardinal invariants composing the above diagrams. Are there any additional ZFC-(in)equalities between these cardinal invariants?
\end{problem}

\begin{problem}
Construct examples distinguishing the pairs of the cardinal characteristics $\ell^{\pm n}$ and $q\ell^{\pm n}$, $\ell^{\mp n}$ and $q\ell^{\mp n}$, $\ell^{\wedge n}$ and $q\ell^{\wedge n}$, $\ell^{\vee n}$ and $q\ell^{\vee n}$.
\end{problem}

We do not know the answer even to

\begin{problem} Is $q\ell^{\pm1}=\ell^{\pm1}$?
\end{problem}

\begin{remark} The cardinal characteristics $\ell^{\pm2}$, $q\ell^{\mp n}$, $q\ell^{\mp n}$, and $q\ell^{\vee n}$ are used in the paper \cite{BR} for evaluation of the submetrizability number and $i$-weight of paratopological groups and topological monoids.
\end{remark}

\section{Foredensity of a topological space}\label{s4}

In this section we give a partial answer to Problem~\ref{pr1} studying the cardinal invariant $\elm=\ell^{\mp1}$ called the {\em foredensity}. Observe that for a topological space $X$ its {\em foredensity} $\elm(X)$ is equal to the smallest cardinal $\kappa$ such that for any neighborhood assignment $V=\bigcup_{x\in X}\{x\}\times O_x$ on $X$ there is a subset $A\subset X$ of cardinality $|A|\le\kappa$ which meets every neighborhood $O_x=B(x;V)$, $x\in X$.

By Proposition~\ref{p1.8}(3), $\elm(X)\le d(X)$. In some cases this inequality turns into the equality.

\begin{theorem}\label{t4.1} For any topological space $X$ we get $\elm(X)\le d(X)\le\elm(X)\cdot\chi(X)$.\newline If $|X|<\aleph_\w$, then $\elm(X)=d(X)$.
\end{theorem}

\begin{proof} Fix a family of neighborhood assignments $(U_\alpha)_{\alpha<\chi(X)}$ on $X$ such that for every $x\in X$ the family $\{B(x;U_\alpha)\}_{\alpha<\chi(X)}$ is a neighborhood base at $x$. By the definition of the foredensity $\elm(X)$, for every $\alpha<\chi(X)$ there is a subset $A_\alpha\subset X$ of cardinality $|A_\alpha|\le\elm(X)$ such that $A_\alpha\cap B(x;U_\alpha)\ne\emptyset$ for every $x\in X$. Then the union $A=\bigcup_{\alpha<\chi(X)}A_\alpha$ is a dense set of cardinality $\elm(X)\cdot\chi(X)$ in $X$, which implies that $d(X)\le|A|\le \elm(X)\cdot\chi(X)$.
\smallskip

Now assume that $|X|<\aleph_\w$. To derive a contradiction, assume that $\elm(X)<d(X)$. First we show that $d(X)$ is infinite. Assuming that $d(X)$ is finite, fix a dense subset $D\subset X$ of cardinality $|D|=d(X)$. We claim that any two distinct points $x,y\in D$ have disjoint neighborhoods.
Assuming that for some distinct points $x,y\in D$ any two neighborhoods $O_x,O_y\subset X$ have non-empty intersection, we conclude that $O_x\cap O_y\cap D\ne\emptyset$. Since $D$ is finite, there is a point $z\in D$ such that $z\in O_x\cap O_y$ for any neighborhoods $O_x$ and $O_y$ of $x$ and $y$, respectively. Then $D'=\{z\}\cup(D\setminus\{x,y\})$ is a dense set of cardinality $|D'|<|D|=d(X)$, which contradicts the definition of $d(X)$. Therefore we can choose a neighborhood assignment $(O_x)_{x\in X}$ such that $O_x\cap O_y=\emptyset$ for any distinct points $x,y\in D$.

By the definition of the foredensity $\elm(X)$, there is a subset $P\subset X$ of cardinality $|P|\le \elm(X)$ such that $O_x\cap P\ne\emptyset$ for every $x\in D$. Taking into account that the family $(O_x)_{x\in D}$ is disjoint, we conclude that $\elm(X)\ge |P|\ge |D|=d(X)$ and this is a desired contradiction showing that $d(X)$ is infinite.

Let $\U$ be the family of all open sets $U\subset X$ with $d(U)\le\elm(X)$. By Proposition~\ref{p1.8}(3), $c(X)\le \elm(X)$. So, we can find a subfamily $\V\subset\U$ of cardinality $|\V|\le c(X)\le \elm(X)$ such that $\bigcup\V$ is dense in the open set $\bigcup\U$. Since $d(\bigcup\U)\le d(\bigcup\V)\le \sum_{V\in\V}d(V)\le |\V|\cdot \elm(X)=\elm(X)^2<d(X)$, the set $\bigcup \U$ is not dense in $X$. So, the set $W=X\setminus \overline{\bigcup\U}$ is not empty and each non-empty open subset $V\subset W$  has density $d(V)>\elm(X)$, which implies that any subset $A\subset W$ of cardinality $|A|\le\elm(X)$ is nowhere dense in $W$. Fix a non-empty open set $V\subset W$ of smallest possible cardinality. Then each non-empty open set $V'\subset V$ has cardinality $|V'|=|V|$.

Consider the cardinals $\mu=\elm(X)$ and $\kappa=|V|$. It follows that $\mu=\elm(X)<d(V)\le|V|=\kappa\le|X|<\aleph_\w$. Consider the family $[V]^\mu=\{A\subset V:|A|\le\mu\}$. By Lemma~5.10(3) in \cite{AM}, $\cf([\kappa]^\mu)=\kappa$, which allows us to find a subfamily $\{C_\alpha\}_{\alpha<\kappa}\subset[V]^\mu$ such that each set $A\in[V]^\mu$ is contained in some set $C_\alpha$, $\alpha<\kappa$. By transfinite induction for every $\alpha<\kappa$ choose a point $x_\alpha\in V\setminus(C_\alpha\cup\{x_\beta\}_{\beta<\alpha})$. The choice of the point $x_\alpha$ is possible since the set $C_\alpha$ is nowhere dense in $V$ and the open set $V_\alpha=V\setminus\overline{C}_\alpha$ is not empty and hence has cardinality $|V_\alpha|=|V|>|\{x_\beta\}_{\beta<\alpha}|$. After completing the inductive construction, choose a neighborhood assignment $U$ on $X$ such that $B(x_\alpha;U)=V_\alpha$ for every $\alpha<\kappa$.
By definition of $\elm(X)$ for the neighborhood assignment $U$ there is a subset $A\subset X$ of cardinality $|A|\le\elm(X)$ such that $B(x_\alpha;U)\cap A\ne\emptyset$ for every $\alpha<\kappa$.
By the choice of the family $\C$ there is an ordinal $\alpha<\kappa$ such that $A\cap V\subset C_\alpha$. Then $B(x_\alpha;U)\cap A=V_\alpha\cap A=(V\setminus{\overline{C}_\alpha})\cap A=\emptyset$, which is a desired contradiction proving the equality $\elm(X)=d(X)$.
\end{proof}

Theorem~\ref{t4.1} will be completed by an example of a topological space (actually, a semitopological group) whose foredensity is strictly smaller than its density. In the proof we shall use one simple fact of the Shelah's pcf-theory (see, \cite{AM}, \cite{BM}). The pcf-theory studies possible cofinalities of ultraproducts of increasing sequences of cardinals. Namely, let $\kappa$ be a singular cardinal, i.e., an uncountable cardinal with $\cf(\kappa)<\kappa$. The cardinal $\kappa$ can be written as $\kappa=\sup_{\alpha\in\cf(\kappa)}\kappa_\alpha$ for a strictly increasing sequence of regular cardinals $(\kappa_\alpha)_{\alpha\in\cf(\kappa)}$. Fix any ultrafilter $U$ on $\cf(\kappa)$ extending the filter $\{\cf(\kappa)\setminus A:|A|<\cf(\kappa)\}$. Such ultrafilters will be called {\em $\cf(\kappa)$-regular}. The ultrafilter $U$ generates the linear preorder $\le_U$ on $\prod_{\alpha\in\cf(\kappa)}\kappa_\alpha$ defined by $f\le_U g$ iff $\{\alpha\in\cf(\kappa):f(\alpha)\le g(\alpha)\}\in U$. By $\cf_U(\prod_\alpha\kappa_\alpha)$ we denote the cofinality of the linearly preordered set $(\prod_{\alpha\in\cf(\kappa)}\kappa_\alpha,\le_U)$. It is equal to the smallest cardinality of a cofinal subset $\C\subset\prod_{\alpha\in\cf(\kappa)}\kappa_\alpha$ (the cofinality of $\C$ means that for any function $f\in\prod_{\alpha\in\kappa}\kappa_\alpha$ there is a function $g\in\C$ such that $f\le_U g$).
We shall need the following known fact (whose proof is included for convenience of the reader).

\begin{lemma}\label{cof} $\cf_U(\prod_\alpha\kappa_\alpha)>\kappa$.
\end{lemma}

\begin{proof} Since the preorder $\le_U$ is linear, the cardinal $\cf_U(\prod_\alpha\kappa_\alpha)$ is regular. Since $\kappa$ is singular, the strict inequality $\cf_U(\prod_\alpha\kappa_\alpha)>\kappa$ will follow as soon as we prove that $\cf_U(\prod_\alpha\kappa_\alpha)\ge\kappa$. To derive a contradiction, assume that $\cf_U(\prod_\alpha\kappa_\alpha)<\kappa$. Then the product $\prod_{\alpha\in\cf(\kappa)}\kappa_\alpha$ contains a cofinal subset $\C$ of cardinality $|\C|<\kappa$. Find an ordinal $\beta<\cf(\kappa)$ such that $|\C|<\kappa_\beta$. For every ordinal $\alpha\in\big[\beta,\cf(\kappa)\big)$, the regularity of the cardinal $\kappa_\alpha$ guarantees that the set $\{f(\alpha):f\in\C\}$ is bounded in the cardinal $\kappa_\alpha=[0,\kappa_\alpha)$. So there exists an ordinal $g(\alpha)\in\kappa_\alpha$ such that $f(\alpha)<g(\alpha)$ for all functions $f\in\C$. Fix a function $\bar g\in\prod_{\alpha\in\cf(\kappa)}\kappa_\alpha$ such that $\bar g(\alpha)=g(\alpha)$ for all $\alpha\in[\beta,\cf(\kappa))$. By the cofinality of $\C$, there is a function $f\in \C$ such that $\bar g\le_U f$. Then the set $\{\alpha\in\cf(\kappa):\bar g(\alpha)\le f(\alpha)\}$ belongs to the ultrafilter $U$, which is not possible as this set is contained in the set $[0,\beta)$, which does not belong to  $U$ by the $\cf(\kappa)$-regularity of $U$. This contradiction completes the proof of the inequality $\cf_U(\prod_\alpha\kappa_\alpha)>\kappa$.
\end{proof}

With Lemma~\ref{cof} in our disposition, we are able to present the promised example of a semitopological group whose foredensity is strictly smaller than its density. We recall that a {\em semitopological group} is a group $G$ endowed with a topology $\tau$ making the group operation $\cdot:G\times G\to G$ separately continuous. This is equivalent to saying that for every elements $a,b\in G$ the two-sided shift $s_{a,b}:G\to G$, $s_{a,b}:x\mapsto axb$, is continuous.

\begin{proposition}\label{p1.9} Let $(G,\tau)$ be a semitopological group of singular cardinality $\kappa$ and density $d(G)<\kappa=|G|$. For any infinite cardinal $\delta\le\kappa$ consider the $T_1$-topology $$\ddot\tau=\{U\setminus A:U\in\tau,\;|A|<\delta\}$$ on $G$. Then $\ddot G=(G,\ddot\tau)$ is a semitopological group of density $d(\ddot G)=d(G)\cdot \delta$ and foredensity $\elm(\ddot G)\le d(G)\cdot \cf(\kappa)$.
\end{proposition}

\begin{proof} It is clear that $\ddot G=(G,\ddot\tau)$ is a semitopological group. Fix any dense subgroup $D\subset G$ of cardinality $|D|=d(G)<\kappa$.

To see that $d(\ddot G)=d(G)\cdot\delta$, fix any subset $E\subset G$ of cardinality $|E|=\delta$  and observe that the set $DE=\{xy:x\in D,\;y\in E\}$ is dense in $\ddot G$ and hence $d(\ddot G)\le|DE|=\max\{d(G),\delta\}=d(G)\cdot \delta$. The reverse inequality $d(\ddot G)\ge\max\{d(G),\delta\}$ follows from the definition of the topology $\ddot\tau$ and the inclusion $\tau\subset\ddot\tau$.
\smallskip

It remains to prove that $\elm(G)\le\max\{d(G),\cf(\kappa)\}$. Fix any neighborhood assignment $V$ on $\ddot G$.
Write the singular cardinal $\kappa$ as $\kappa=\sup_{\alpha<\cf(\kappa)}\kappa_\alpha$ for some strictly increasing sequence of regular cardinals $(\kappa_\alpha)_{\alpha\in\cf(\kappa)}$ such that $\kappa=\sup_{\alpha\in\cf(\kappa)}\kappa_\alpha$.
Let $U$ be any $\cf(\kappa)$-regular ultrafilter on the cardinal $\cf(\kappa)$.

Fix any bijective map $\xi:G\to\kappa$. For every $x\in G$ by the definition of the topology $\ddot\tau$ there exists a neighborhood $O_x\subset G$ of $x$ in the topology $\tau$ such that $|O_x\setminus B(x;V)|<\kappa$. Then the set $D_x=D\cdot\big(O_x\setminus B(x;V)\big)$ has cardinality $|D_x|\le|D|\cdot |O_x\setminus B(x;U)|<\kappa$ and hence $|D_x|<\kappa_{\alpha_x}$ for some ordinal $\alpha_x\in\cf(\kappa)$. Using the regularity of the cardinals $\kappa_\beta$, $\beta\in[\alpha_x,\cf(\kappa){[}\,$, we can choose a function $f_x\in\prod_{\alpha\in\cf(\kappa)}\kappa_\alpha$ such that for any ordinal $\beta\in[\alpha_x,\cf(\kappa){[}$ \ we get $[0,\kappa_\beta{[}\,\cap\,\xi(D_x)\subset [0,f_x(\beta){[}\,$. By Lemma~\ref{cof}, the set of functions $\{f_x:x\in G\}$ is not cofinal in the linearly preordered set $(\prod_{\alpha\in\cf(\kappa)}\kappa_\alpha,\le_U)$. Consequently, there exists a function $f\in \prod_{\alpha\in\cf(\kappa)}\kappa_\alpha$ such that $f\not\le_U f_x$ for every $x\in G$.
Consider the set $F=\{\xi^{-1}(f(\alpha)):\alpha\in\cf(\kappa)\}\subset G$. To finish the proof it remains to check that the set $DF$ meets each $U$-ball $B(x;V)$, $x\in G$.

To derive a contradiction, assume that $DF\cap B(x;V)=\emptyset$ for some point $x\in G$.
We claim that $F\subset D_x$. Assuming that $F\not\subset D_x=D\cdot(O_x\setminus B(x;V))$ and taking into account that $D$ is a subgroup of $G$, we could find a point $y\in F$ such that $y\notin D\cdot(O_x\setminus B(x;V))$, which implies $Dy\cap (O_x\setminus B(x;V))=\emptyset$ and hence $Dy\cap O_x\subset B(x;V)$. Using the density of the sets $D$ and $Dy$ in $(G,\tau)$, find a point $z\in D$ with $zy\in O_x$. Then $zy\in Dy\cap O_x\subset B(x;V)$, which contradicts $zy\in DF\cap B(x;V)=\emptyset$.

So, $F\subset D_x$ and hence $[0,\kappa_\beta{[}\,\cap\,\xi(F)\subset[0,\kappa_\beta{[}\,\cap\,\xi(D_x)\subset [0,f_x(\beta){[}\,$ for any $\beta\in[\alpha_x,\cf(\kappa){[}\,$. It follows from $f\not\le_U f_x$ that the set $\{\alpha\in\cf(\kappa):f(\alpha)\le f_x(\alpha)\}$ does not belong to the ultrafilter $U$ and hence the set $U_x=\{\alpha\in\cf(\kappa):f(\alpha)>f_x(\alpha)\}$ belongs to $U$. The $\cf(\kappa)$-regularity of the ultrafilter $U$ guarantees that $|U_x|=\cf(\kappa)$ and hence we can find an ordinal $\beta\in U_x$ with $\beta>\alpha_x$. For this ordinal we get $f(\beta)\in \xi(F)\cap [0,\kappa_\beta{[}\,\subset [0,f_x(\beta){[}\,$ which contradicts $f_x(\beta)<f(\beta)$. This contradiction shows that the set $DF$ of cardinality $|DF|\le d(G)\cdot\cf(\kappa)$ meets every $V$-ball $B(x;V)$, $x\in G$, witnessing that $\elm(G)\le d(G)\cdot\cf(\kappa)$.
\end{proof}

A topological space $X$ is called {\em totally disconnected} if for any distinct points $x,y\in X$ there is a closed-and-open subset $U\subset X$ such that $x\in U$ and $y\notin U$. Observe that each totally disconnected space is functionally Hausdorff.

\begin{corollary}\label{c1.10} For any singular cardinal $\kappa$ (with $\kappa\le 2^{2^{\cf(\kappa)}}$) there exists a (totally disconnected) $T_1$-space $X$ such that $\elm(X)=\cf(\kappa)$ and $d(X)=|X|=\kappa$.
\end{corollary}

\begin{proof} Given a singular cardinal $\kappa$, consider the power $\IZ_2^{2^{\cf(\kappa)}}$ of the two-element group $\IZ_2=\IZ/2\IZ$. Endowed with the Tychonoff product topology,  $\IZ_2^{2^{\cf(\kappa)}}$ is a compact Hausdorff topological group.  By Hewitt-Marczewski-Pondiczery Theorem \cite[2.3.15]{Eng}, this topological group has density $d(\IZ_2^{2^{\cf(\kappa)}})=\cf(\kappa)$.
So, $\IZ_2^{2^{\cf(\kappa)}}$ contains a dense subgroup $G_0$ of cardinality $|G_0|=\cf(\kappa)$.

If $\kappa\le 2^{2^{\cf\kappa}}$, then we can choose a subgroup $G\subset\IZ_2^{2^{\cf(\kappa)}}$ of cardinality $|G|=\kappa$ with $G_0\subset G$. Let $\tau$ be the topology on $G$ inherited from $\IZ_2^{2^{\cf(\kappa)}}$. It is clear that $(G,\tau)$ is a totally disconnected Hausdorff topological group of density $d(G)\le|G_0|=\cf(\kappa)$.

If $\kappa>2^{2^{\cf\kappa}}$, then take any group $G$ of cardinality $|G|=\kappa$ and let $\tau=\{\emptyset,G\}$ be the anti-discrete topology on $G$. In this case $G$ is a topological group of density $d(G)=1$.

Denote by $\ddot G$ the group $G$ endowed with the topology
$$\ddot\tau=\{U\setminus A:U\in\tau,\;|A|<\kappa\}.$$ By Proposition~\ref{p1.9}, the topological $T_1$-space  $\ddot G$ has density $d(\ddot G)=|\ddot G|=\kappa$ and foredensity $\elm(\ddot G)\le \cf(\kappa)$. If $\kappa\le 2^{2^{\cf(\kappa)}}$, then the topology $\tau$ is totally disconnected and Hausdorff, and so is the topology $\ddot \tau$.
\end{proof}

Observe that under Generalized Continuum Hypothesis no singular cardinal $\kappa$ satisfies the inequality $\kappa\le 2^{2^{\cf(\kappa)}}$. In this case (a part of) Corollary~\ref{c1.10} is vacuous.
This suggests the following problem.

\begin{problem}\label{pr1} Is there a ZFC-example of a Hausdorff space $X$ with $\elm(X)<d(X)$?
\end{problem}

Another natural problem concerns (non)regularity of spaces constructed in Proposition~\ref{p1.9}.

\begin{problem}\label{pr2} Is there a regular $T_1$-space $X$ with $\elm(X)<d(X)$?
\end{problem}

\begin{remark} Answering Problems~\ref{pr1}, \ref{pr2} Juh\'asz, Soukup and Szentmikl\'ossy  \cite{JSS} proved the equivalence of the following statements:
\begin{enumerate}
\item $2^\kappa<\kappa^{+\w}$ for each cardinal $\kappa$;
\item $\elm(X)=d(X)$ for each Hausdorff space $X$;
\item $\elm(X)=d(X)$ for each zero-dimensional (and hence regular) Hausdorff space $X$.
\end{enumerate}
\end{remark}

\end{document}